\documentclass{amsart}

\usepackage{cite}

\usepackage{tikz-cd}

\usepackage{amsmath, amssymb, amsthm, amscd}

\newcommand{\R}  {{\mathbb R}}
\newcommand{\C}  {{\mathbb C}}
\newcommand{\K} {{\mathbb K}}
\newcommand{\N}  {{\mathbb N}}
\newcommand{\Z}  {{\mathbb Z}}
\newcommand{\eps}{\varepsilon}

\newcommand{\bx}{{\mathbf x}}
\newcommand{\by}{{\mathbf y}}

\newcommand{\bnu}{{\boldsymbol{\nu}}}
\newcommand{\NN}{{\boldsymbol{N}}}

\newcommand{\gjc}{G_j^c}
\newcommand{\fjc}{F_j^c}
\newcommand{\na}{\mathrm{I}}
\newcommand{\nb}{\mathrm{I\!I}}

\newcommand{\spann}{\operatorname{span}}

\newcommand{\Int}  {\operatorname{Int}} 
\newcommand{\App}  {\operatorname{App}} 
\newcommand{\aall} {\mathcal{A}^{\mathrm{all}}}
\newcommand{\astd} {\mathcal{A}^{\mathrm{std}}}
\newcommand{\bb}   {\mathcal{B}}
\newcommand{\bstd} {\bb^{\mathrm{std}}}
\newcommand{\Act}  {\operatorname{Act}}

\newcommand{\scp}[2]{\langle #1, #2 \rangle}

\DeclareMathOperator{\decay}{decay}
\DeclareMathOperator{\dec}{dec}
\DeclareMathOperator{\err}{err}
\DeclareMathOperator{\cost}{cost}

\theoremstyle{plain}
\newtheorem{lemma}{Lemma}[section]
\newtheorem{theo}[lemma]{Theorem}
\newtheorem{cor}[lemma]{Corollary}
\theoremstyle{definition}
\newtheorem{rem}[lemma]{Remark}
\newtheorem{exmp}[lemma]{Example}

\begin{document}

\title[Integration and Approximation with Increasing Smoothness]
{Embeddings for Infinite-Dimensional
Integration and $L_2$-Approximation 
with Increasing Smoothness}

\author[Gnewuch]
{M.~Gnewuch}
\address{
Mathematisches Seminar\\
Christian-Albrechts-Universit\"at zu Kiel\\
Ludewig-Meyn-Str.\ 4\\
24098 Kiel\\ 
Germany}
\email{gnewuch@math.uni-kiel.de}

\author[Hefter]
{M.~Hefter }
\address{Fachbereich Mathematik\\
Technische Universit\"at Kaisers\-lautern\\
Postfach 3049\\
67653 Kaiserslautern\\
Germany}
\email{hefter@mathematik.uni-kl.de}

\author[Hinrichs]
{A.~Hinrichs}
\address{
Institut f\"ur Analysis\\
Johannes-Kepler-Universit\"at Linz\\
Altenberger Str.\ 69\\
4040 Linz\\
Austria}
\email{aicke.hinrichs@jku.at}

\author[Ritter]
{K.~Ritter}
\address{Fachbereich Mathematik\\
Technische Universit\"at Kaisers\-lautern\\
Postfach 3049\\
67653 Kaiserslautern\\
Germany}
\email{ritter@mathematik.uni-kl.de}

\author[Wasilkowski]
{G.~W.~Wasilkowski}
\address{Department of Computer Science\\
Davis Marksbury Building\\
329 Rose St. \\
University of Kentucky\\
Lexington, KY 40506-0633, USA}
\email{greg@cs.uky.edu}

\date{September 17, 2018}


\begin{abstract}
We study integration and $L_2$-approximation 
on countable tensor products of function spaces of increasing smoothness.
We obtain upper and lower bounds for the minimal errors, which
are sharp in many cases including, e.g., Korobov, Walsh, Haar, 
and Sobolev spaces. 
For the proofs we derive embedding theorems between 
spaces of increasing smoothness and 
appropriate weighted function spaces of fixed smoothness. 
\end{abstract}

\maketitle

\section{Introduction}

We study integration and $L_2$-approximation 
for functions of infinitely many variables.
The complexity of computational problems of this kind has first
been analyzed in \cite{HW01,HMNR10,KuoEtAl10}; for further
contributions we refer to, e.g., 
\cite{BG12, DG14, DG13, DunGri16,Gne12a, GneEtAl16, Hinrichs2018, MR3507353,PW11,Was10,Was12,Was13,MR2805529}.
First of all, this line of research may be viewed 
as the limit of tractability analysis of multivariate problems,
where the number of variables tends 
to infinity. Furthermore, computational problems with infinitely many 
variables naturally arise in a number of different applications. One
example are stochastic differential equations,
since the driving processes, often a finite- or
infinite-dimensional Brownian motion, is canonically represented in 
terms of a sequence of independent and identically distributed
random variables. Another example are
partial differential equations with random coefficients, where
similar representations are employed for the underlying 
random fields. 

Roughly speaking, problems with a large or infinite number of variables
are computationally tractable if the variables may be
arranged in such a way that their impact decays sufficiently fast. 

The first, and still most popular approach to capture this phenomenon
are weighted function spaces, where the weights
directly moderate the influence of groups of variables.
We refer to \cite{SW98} as the pioneering paper and, e.g.,
to \cite{DKS13, NovWoz08, NovWoz10, NovWoz12}
for further results and references in the multivariate case.
For problems with infinitely many variables weighted function spaces
have first been studied in \cite{HW01}, and the structure of the
corresponding spaces is analyzed in \cite{GMR14}. See, e.g.,
\cite{GneEtAl16} for recent results and references on 
infinite-dimensional integration.

As an alternative concept, an increasing smoothness with respect
to the properly ordered variables has first been studied in
tractability analysis in
\cite{PapWoz10}, and further results in this setting have been
derived in, e.g.,
\cite{DunGri16,HHPS18,KriPiWo14,IKPW16a,Sie14}.
We add that this kind of smoothness phenomenon is present for most
of the partial differential equations with random coefficients
that have been studied in the literature from a
computational point of view, see \cite{DunGri16,HHPS18} for further
information.
Moreover, increasing smoothness is a particular instance of anisotropic
smoothness, as studied in approximation theory, see, e.g.,
\cite[Sec.~10.1]{DTU18} for further information.

The function spaces under consideration in the present paper are
tensor products
\[
H := \bigotimes_{j \in \N} H_j
\]
for scales of Hilbert spaces $H_j$ of functions of a
single variable, defined on any domain $D$. 
Accordingly, the elements of $H$ are functions on the domain
$E := D^\N$. 
For integration and $L_2$-approximation
the underlying probability measure $\mu$ on $E$ is 
the countable product of an arbitrary 
probability measure $\mu_0$ on $D$.

Originally, we are interested in the case of
spaces $H_j$ of increasing smoothness in the sense that
\[
H_1 \supset H_2 \supset \dots
\]
with compact embeddings. 
The main aim of this paper is to show that this setting
may be reduced to tensor products
of suitable weighted function spaces $H_j$
via embeddings. 
Reductions of this type lead to sharp upper and lower bounds for minimal
errors for integration and $L_2$-approximation, despite the fact
that the weighted spaces $H_j$
are isomorphic as Banach spaces, 
while we have compact embeddings in the case 
of increasing smoothness.

The embeddings between the two kinds of rather different
tensor product spaces allow to derive new results for 
tensor products of spaces of increasing smoothness from known results for 
tensor products of weighted spaces that have a fixed smoothness. 
We carry out this program for
Korobov spaces, Walsh spaces, Haar spaces, and Sobolev spaces of functions 
with derivatives in weighted $L_2$-spaces.

The embedding approach, which has first been developed in
\cite{MR3325681}, has meanwhile been applied to a number
of different settings also beyond the Hilbert space 
and the tensor product case, see 
\cite{GneEtAl16,Hinrichs2018,MR3424634,MR3631818,MR3574544,
MR3507353,MR3450729}.
Embeddings between
spaces of increasing smoothness and weighted function spaces
have first been observed and exploited in \cite{Leo15}.

For integration we wish to approximate  
$\int_E f \, d \mu$ 
for $f \in H$, and for $L_2$-approxima\-tion we wish 
to recover $f\in H$ with error measured in $L_2(E,\mu)$.
We are primarily interested in algorithms that use standard information, 
i.e., algorithms that may only use 
a finite number of function values of any $f$,
which requires $H$ to be a reproducing kernel Hilbert space.

Since the functions $f \in H$ depend on infinitely-many variables, it is 
unreasonable to assume that they may be evaluated at any point 
$\by \in E$ at unit cost. Instead we employ the so-called
unrestricted subspace sampling model, which has been introduced
in \cite{KuoEtAl10}. For a fixed nominal value $a \in D$ function
values are only available at points $\by = (y_j)_{j \in \N} \in E$
with 
\[
\Act(\by) :=\#\{j\in\N\colon y_j\neq a\} < \infty,
\]
and $\Act(\by)$ (or a function thereof) is 
the cost of function evaluation at such an admissible point $\by$.
Accordingly, the cost of a linear deterministic algorithm
\[
A(f) = \sum_{i=1}^m f(\by_i) \cdot z_i
\]
with admissible points $\by_i \in E$ and with scalars $z_i$ for
integration and $z_i \in L_2(E,\mu)$ for $L_2$-approximation 
is given by $\cost(A) := \sum_{i=1}^m \Act (\by_i)$.

The key quantities in the worst case analysis on the unit ball
$B(H) \subset H$ are the
$n$-th minimal errors
\begin{align*}
\err_n(H,\Int,\astd)
:=
\inf_{\cost(A) \leq n} \sup_{f \in B(H)} 
\left|\int_E f \, d \mu - A(f)\right|
\end{align*}
for integration and
\begin{align*}
\err_n(H,\App,\astd)
:=
\inf_{\cost(A) \leq n} \sup_{f \in B(H)} \|f-A(f)\|_{L_2(E,\mu)} 
\end{align*}
for $L_2$-approximation.

Let us describe the function space setting in more detail.
We focus on scales of function spaces $H_j$
with the following structure, 
later on called the standard setting, 
which is based on an orthonormal basis 
$(e_\nu)_{\nu \in \N_0}$ of $H_0 := L_2(D,\mu_0)$
with $e_0=1$ and on a family $(\alpha_{\nu,j})_{\nu,j \in \N}$ 
of positive Fourier weights.
With $\scp{\cdot}{\cdot}_0$ denoting the scalar product on $H_0$,
we define $H_j$ to be the Hilbert space of all $f \in H_0$ such that
\[
\|f\|_j^2 := 
|\scp{f}{e_0}_0|^2 +
\sum_{\nu \in \N} \alpha_{\nu,j} \cdot |\scp{f}{e_\nu}_0|^2 
< \infty.
\]
Typically, the asymptotic properties of the Fourier weights
ensure that $(H_j)_{j \in \N_0}$ is a scale
of spaces of increasing smoothness.
In any case, $H \subseteq L_2(E,\mu)$ by assumption.

To give a flavor of our results, let us consider the uniform distribution 
$\mu_0$ on $D := [0,1]$ and the trigonometric basis 
given by
$e_\nu(x):=\exp(2\pi i (-1)^\nu \lceil \nu/2 \rceil x)$, together
with the Fourier weights
\[
a_{\nu,j} := \left(1 + \lfloor (\nu+1)/2 \rfloor\right)^{r_j},
\]
where
\[
0 < r_j < r_{j+1}
\]
for all $j\in\N$.
In this case the space $H_j$ is the Korobov space with
smoothness parameter $r_j$.
As a well-known fact, $H_j$ is a 
reproducing kernel Hilbert space if and only if $r_j > 1$, and for
an even integer $r_j \geq 2$ the elements of $H_j$ have a weak 
derivative of order $r_j/2$ in $L_2(D,\mu_0)$.
Given $r_1 > 1$, 
\[
\rho := \liminf_{j \to \infty} \frac{r_j}{\ln(j)} > \frac{1}{\ln(2)}
\]
is a sufficient condition for $H$ to be a reproducing kernel Hilbert 
space of functions on the domain $[0,1]^\N$. A necessary condition
also permits $\rho = 1/\ln(2)$.
See Example~\ref{ex11} and \ref{ex1}.
We determine the decay of the $n$-th minimal error for $S = \Int$
and $S=\App$ in Corollary~\ref{c2}.
It turns out that
this decay is equal to
\begin{align*}
\dec = \tfrac{1}{2}\cdot \min(r_1,\rho\cdot \ln(2)-1)
\end{align*}
for both problems, i.e., for every $\varepsilon>0$ there exists a 
constant $c>0$ such that 
\begin{align}\label{einl4}
\err_n(H,S,\astd)
\leq c\cdot n^{-(\dec-\varepsilon)}
\end{align}
for all $n\in\N$, and $\dec$ is minimal with this property.
We observe, in particular, that the minimal smoothness $r_1$
with respect to a single variable and
the increase of the smoothness along the variables,
as quantified by $\rho$, are the crucial parameters:
together they determine whether $H$ is a reproducing kernel Hilbert
space as well as the asymptotic behavior of the $n$-th minimal
errors. 

Let us provide some details of our proof strategy, which applies to the 
standard setting in general,
see Section~\ref{sec35} for the embeddings
and Section~\ref{sec42} for the results on integration and approximation.
The reproducing kernel $K$ of the Hilbert space $H = H(K)$ is
the tensor product
\[
K := \bigotimes_{j \in \N} k_j
\]
of the reproducing kernels $k_j$ of the spaces $H_j = H(k_j)$, 
see Section \ref{s2.2}.
For the proof of the upper bound \eqref{einl4},
we determine a sequence of weights $\theta_j > 0$, as small as possible,
and show the existence of a reproducing kernel $m$ for
functions of a single variable
with the following properties, see Theorem \ref{t1}.
The space $H(K)$ is continuously embedded into the Hilbert space
$H(M)$ with reproducing kernel
\[
M := \bigotimes_{j \in \N} (1 + \theta_j \cdot m),
\]
and $H(1+m) = H(k_1)$ as vector spaces. Furthermore, $m$ is
anchored at a given point $a \in D$, i.e., $m(a,a)=0$.
It follows that,
$\err_n(H(K),S,\astd)$ is at most of the order of 
$\err_n(H(M),S,\astd)$.
In this way we relate the tensor product space $H(K)$ of spaces of
increasing smoothness to the tensor product space $H(M)$, which is
based on weighted anchored kernels.
A reverse embedding with a two-dimensional space $H(1+\ell) \subset
H(k_1)$ is part of the proof that $\dec$ is maximal with
the property~\eqref{einl4}.

Integration and $L_2$-approximation is thoroughly studied in 
the literature for tensor products of weighted anchored
spaces, where the multivariate decomposition method
has been established as a powerful generic algorithm, see,
e.g., \cite{GilEtAl2017}.
In particular, it is known in this setting
how the asymptotic behavior of
the minimal errors depends on summability properties of
the sequence $(\theta_j)_{j \in \N}$ of weights
and on the minimal errors for the univariate problem,
see \cite{PW11,Was12,Gne12a}.
Interestingly, we obtain sharp results via embeddings in this way, 
although $H(1+\theta_j \cdot m) = H(k_1)$ as vector spaces for
every $j \in \N$, so that we embed $H(K)$ into the much larger space 
$H(M)$ as we trade increasing smoothness for decaying weights.

This paper is organized as follows.
In Section~\ref{s2} we determine when a Hilbert space may be
canonically identified with a reproducing 
kernel Hilbert space; here subspaces of $L_2$-spaces and
countable tensor products are particularly relevant for the 
present paper.
In Sections~\ref{s3.1} and \ref{s3.1a} we present the  
function space framework to introduce and study 
Hilbert spaces of increasing smoothness. Classical examples
are given by
Korobov spaces, Walsh spaces, Haar spaces, and
Sobolev spaces with derivatives in weighted
$L_2$-spaces, see Section \ref{s3.2}.
In Sections~\ref{secembed} and \ref{sec35}
we construct the appropriate tensor products of weighted
(anchored) spaces and provide 
embedding theorems between these spaces and tensor products of
spaces of increasing smoothness.
The embeddings are applied in Section~\ref{s4}
to determine the decay of the minimal errors for
integration and $L_2$-approximation. For the latter problem
we actually compare two classes of algorithms that
may either use standard information, as outlined
above, or, potentially more powerful, use arbitrary bounded
linear functionals at cost one. 
In Appendix~\ref{a1} we recall basic properties of countable tensor 
products of Hilbert spaces, and Appendix~\ref{a2} contains some 
facts on summability and decay of sequences of 
real numbers.
In Appendix~\ref{a3} we consider $L_2$-approximation
in Haar spaces of functions of a single variable.

\pagebreak

\section{Tensor Products and Reproducing Kernels}\label{s2}

\subsection{Reproducing Kernels}\label{s2.0}

Consider a separable Hilbert space $(\mathcal H,\scp{\cdot}{\cdot})$
over $\K\in\{\R,\C\}$
with an orthonormal basis $(h_\nu)_{\nu \in N}$ for some
countable set $N$. Moreover, let 
$E \neq \emptyset$ be any set. For any
injective linear mapping $\Phi:\mathcal{H} \to \K^E$
we define a scalar product $\scp{\cdot}{\cdot}_{\Phi}$ on 
$\Phi(\mathcal{H})$ by 
\[
\scp{\Phi f}{\Phi g}_{\Phi}:= \scp{f}{g}
\]
for all $f,g \in \mathcal{H}$. In this way
we may identify the (abstract) Hilbert space $\mathcal{H}$ with the 
Hilbert space $\Phi(\mathcal{H})$ of real- or complex-valued functions on 
the domain $E$.

The following two lemmata provide a necessary and a sufficient
condition for the function
space $\Phi(\mathcal{H})$ to be a reproducing kernel Hilbert space.

\begin{lemma}\label{l3a}
Suppose that $\Phi:\mathcal{H} \to \K^E$ is linear and injective.
Moreover, assume that
$(\Phi(\mathcal{H}), \scp{\cdot}{\cdot}_{\Phi})$ is 
a reproducing kernel Hilbert space. Then we have
\begin{equation}\label{g3a}
\forall\,y \in E: \quad
\sum_{\nu\in N} |\Phi h_\nu (y)|^2 < \infty
\end{equation}
and
\begin{equation}\label{g21}
\forall\,y \in E\ \forall\, f \in \mathcal{H}: \quad
\Phi f (y) = \sum_{\nu \in N} \scp{f}{h_\nu} \cdot \Phi h_\nu (y)
\end{equation}
with absolute convergence.
Furthermore, the reproducing kernel $K$ of this space is given by
\begin{equation}\label{g22a}
K (x,y) 
= 
\sum_{\nu\in N} 
\Phi h_\nu (x) \cdot \overline{\Phi h_\nu (y)}
\end{equation}
with absolute convergence for all $x,y \in E$.
\end{lemma}

\begin{proof}
For every $f \in \mathcal{H}$ we have
\[
\Phi f = 
\sum_{\nu \in N} 
\scp{\Phi f}{\Phi h_\nu}_{\Phi} \cdot \Phi h_\nu =
\sum_{\nu \in N}
\scp{f}{h_\nu} \cdot \Phi h_\nu
\]
with convergence in $\Phi(\mathcal{H})$. By assumption, point evaluations
are continuous on the latter space, which yields \eqref{g21}.
In particular, for $\Phi f = K(\cdot,y)$ with $y \in E$ we obtain
\[
K(\cdot,y) = 
\sum_{\nu \in N} 
\scp{K(\cdot,y)}{\Phi h_\nu}_{\Phi} \cdot \Phi h_\nu 
= \sum_{\nu\in N} \overline{\Phi h_\nu (y)} \cdot \Phi h_\nu,
\]
which yields \eqref{g22a}. Choose $x:=y$ to derive \eqref{g3a}
from \eqref{g22a}. 
The Cauchy-Schwarz inequality and \eqref{g3a} guarantee the absolute 
convergence in \eqref{g21} and \eqref{g22a}.
\end{proof}

Every mapping $\Phi$ that leads to a reproducing kernel Hilbert
space $\Phi(\mathcal{H})$ is already determined by the values $\Phi h_\nu$ 
for $\nu \in N$, see \eqref{g21}.
In the construction of such a mapping we therefore
start with an injective mapping
$\Phi : \{h_\nu: \nu \in N\} \to \K^E$, and we assume
that \eqref{g3a} is satisfied. The mapping $\Phi$ is extended to a linear 
mapping
$\Phi : \mathcal{H}\to \K^E$
by
\begin{equation}\label{g40}
\Phi f (y) := \sum_{\nu \in N} \scp{f}{h_\nu} \cdot \Phi h_\nu (y).
\end{equation}
Assumption \eqref{g3a} yields the absolute 
convergence of the right-hand side in \eqref{g40} for all 
$f \in \mathcal{H}$ 
and $y \in E$. Actually we have
\begin{equation}\label{g29}
\sum_{\nu \in N} |\scp{f}{h_\nu} \cdot \Phi h_\nu(y)| 
\leq
\left( \sum_{\nu \in N} 
|\scp{f}{h_\nu}|^2 \right)^{1/2} \cdot 
\left( \sum_{\nu \in N} |\Phi h_\nu(y)|^2  \right)^{1/2}
\end{equation}
for $f \in \mathcal{H}$ and $y \in E$.

\begin{lemma}\label{l3b}
Suppose that \eqref{g3a} is satisfied and that $\Phi$ given by
\eqref{g40} is injective. Then 
$(\Phi(\mathcal{H}), \scp{\cdot}{\cdot}_{\Phi})$ is a
reproducing kernel Hilbert space.
\end{lemma}

\begin{proof}
Let $\|\cdot\|_{\Phi}$ denote the norm that 
is induced by $\scp{\cdot}{\cdot}_{\Phi}$.
Observe that
\[
\|\Phi f\|_{\Phi} = 
\left( \sum_{\nu \in N} |\scp{f}{h_\nu}|^2 \right)^{1/2}
\] 
for $f \in \mathcal{H}$. Use \eqref{g3a} and \eqref{g29} to conclude that 
$\Phi f \mapsto \Phi f (y)$
defines a bounded linear functional on $\Phi (\mathcal{H})$ for every $y
\in E$.
\end{proof}

\begin{rem}\label{r60}
In general, \eqref{g3a} does not imply that $\Phi$ defined according 
to \eqref{g40} is injective. An obvious necessary assumption is that 
the set  $\{\Phi h_\nu: \nu \in N\}$ is linearly independent in $\K^E$.
The following example shows that even this is not sufficient.

Let $N := \N$, and let $\mathcal{H} := \ell_2$ with the canonical unit 
vector basis $(h_\nu)_{\nu\in\N}$ and define 
$\Phi : \{h_\nu: \nu \in \N\} \to \K^\N$ by 
$\Phi h_\nu := \nu(h_\nu - h_{\nu-1})$ for $\nu\in\N$ 
with the convention $h_0:=0$. 
For each $y\in \N$ the sum in \eqref{g3a} is a finite sum, 
so \eqref{g3a} is satisfied. 
It is also easy to see that $\{\Phi h_\nu: \nu \in \N\}$ 
is linearly independent in $\K^\N$.
For $f \in \mathcal{H}$ with $\scp{f}{h_\nu}=\frac{1}{\nu}$ and $y\in \N$
we obtain from \eqref{g40} that
\[
\Phi f (y) = \sum_{\nu \in \N} \scp{f}{h_\nu} \cdot \Phi h_\nu (y) = 0.
\]
Hence $\Phi$ is not injective.
\end{rem}

\begin{rem}\label{r61}
Assume that $\Phi$ is linear and injective.
Then, in general, condition \eqref{g3a} is not sufficient to 
guarantee that 
$(\Phi(\mathcal{H}),\scp{\cdot}{\cdot}_{\mathcal{H}})$ 
is a reproducing kernel Hilbert space. 
We present a general counterexample.

We start with a reproducing kernel 
Hilbert space $\mathcal{H} \subseteq \K^{E}$ with an orthonormal basis 
$(h_\nu)_{\nu \in \N}$ and consider 
$\Phi: \mathcal{H} \to \K^{E}$, $f \mapsto f$.
Due to Lemma \ref{l3a} we have
\[
\forall\,y \in E: \quad
\sum_{\nu\in \N} |h_\nu (y)|^2 < \infty.
\]
Let $E_* = E \cup \{z \}$ with a point $z \notin E$.
Choose an arbitrary discontinuous linear functional 
$\zeta$ on $\mathcal{H}$
satisfying $\zeta(h_\nu)=0$ for all $\nu\in \N$. 
Let $\Psi f\in \K^{E_*}$ be the extension 
of $f \in \mathcal{H}$ to $E_*$ 
with $\Psi f(z) = \zeta(f)$.
Obviously, $\Psi$ is a linear and injective mapping from
$\mathcal{H}$ to $\K^{E_*}$. It follows that
$\mathcal{H}_* := \Psi(\mathcal{H})$, equipped with the scalar
product $\langle \cdot, \cdot \rangle_\Psi$ induced by $\Psi$, 
is a Hilbert space, too,
with orthonormal basis $(\Psi h_\nu)_{\nu \in \N}$. 
Since $\Psi h_\nu ( z) = 0$ for 
all $\nu \in \N$, we have 
\[
\forall\,y \in E_*: \quad
\sum_{\nu\in \N} |\Psi h_\nu (y)|^2 < \infty.
\]
But $\mathcal{H}_*$ is not a reproducing kernel Hilbert space since, 
by construction, the function evaluation 
$\Psi f \mapsto \Psi f ( z) = \zeta(f)$
is discontinuous on $\mathcal{H}_*$.  

Notice that $\Psi$ is not of the form \eqref{g40}. Indeed,
$\sum_{\nu \in \N} \langle f, h_\nu \rangle \Psi h_\nu(z) =0$ for all 
$f\in \mathcal{H}$, but since $\zeta$ is discontinuous, there has 
to exist at least one $g \in \mathcal{H}$ satisfying 
$\Psi g(z) = \zeta(g) \neq 0$.
\end{rem}

\begin{rem}\label{r51}
The particular case where $\mathcal{H}$ already consists of real- or
complex-valued functions on $E$ with the natural choice of 
$\Phi f := f$ for every $f \in \mathcal{H}$ is 
also studied in
\cite[Rem.~1]{IKPW16a}. It is shown that $\mathcal{H}$ is a reproducing
kernel Hilbert space if and only if \eqref{g3a} and \eqref{g21} are 
satisfied.
\end{rem}

Lemma \ref{l3b} allows to go beyond the setting from Remark
\ref{r51} in order to cover the most important case of $\mathcal{H}$ being
a subspace of an $L_2$-space.
Here it turns out that \eqref{g3a} already
implies that the pointwise limits of the Fourier partial sums form a 
reproducing kernel Hilbert space.

\begin{rem}\label{r50}
Consider the space $L_2(E,\mu)$ with respect to
any measure $\mu$ on any $\sigma$-algebra on $E$, and assume that 
$\mathcal{H}$ is a linear subspace of $L_2(E,\mu)$ with a continuous 
embedding. Consider a sequence of square-integrable 
functions $\mathfrak{h}_\nu$ on $E$ with the following properties:
The corresponding equivalence classes $h_\nu \in L_2(E,\mu)$ form an
orthonormal basis of $\mathcal{H}$, and
\[
\forall\,y \in E: \quad
\sum_{\nu\in N} |\mathfrak{h}_\nu (y)|^2 
< \infty,
\]
cf.~\eqref{g3a}.
We claim that $\Phi$ given by \eqref{g40} with 
$\Phi h_\nu := \mathfrak{h}_\nu$, i.e.,
\[
\Phi f (y) := 
\sum_{\nu \in N} \scp{f}{h_\nu} \cdot \mathfrak{h}_\nu(y),
\]
is injective. 

In fact, consider a square-integrable function $\mathfrak{f}$ on
$E$, whose corresponding equivalence class $f \in L_2(E,\mu)$ 
satisfies $f \in \mathcal{H}$ and $\Phi f =0$.
The partial sums of the series
$\sum_{\nu \in N} \scp{f}{h_\nu} \cdot \mathfrak{h}_\nu$
converge in mean-square to $\mathfrak{f}$. Due 
to the Fischer-Riesz Theorem there exists a subsequence
of partial sums
that converges almost everywhere to $\mathfrak{f}$. Since 
$\Phi f=0$ means that the partial sums converge to zero at every
point in $E$, we get $\mathfrak{f}=0$ almost everywhere,
i.e., $f=0$.

Apply Lemma \ref{l3b} to conclude that 
$(\Phi(\mathcal{H}), \scp{\cdot}{\cdot}_{\Phi})$ is a 
reproducing kernel Hilbert space. 
We add that the inverse 
$\Phi^{-1}\colon \Phi(\mathcal{H})\to L_2(E,\mu)$ of $\Phi$ 
is continuous and maps $\mathfrak{f}\in \Phi(\mathcal{H})$ to its 
equivalence class.
\end{rem}

\subsection{Countable Tensor Products}\label{s2.2}

Consider a sequence of separable Hilbert spaces
$(H_j,\scp{\cdot}{\cdot}_j)$ with $j \in \N$ together
with orthonormal bases $(h_{\nu,j})_{\nu \in N_j}$ with countable
sets $N_j$. For notational convenience assume that $N_j \subseteq \N_0$ 
and $0 \in N_j$.
Later on we will have $N_j = \N_0$ for all $j \in \N$ most of the
time. However, we also consider the
case $N_j = \{0,1\}$ for all $j \in\N$.

The countable tensor product
\[
H := \bigotimes_{j \in \N} H_j
\]
that is studied in this paper is the so-called incomplete tensor product 
introduced by von Neumann in \cite{Neu39}
with the particular choice of the unit vector 
$h_{0,j}$ in the space $H_j$.
The choice of $\nu=0$ is without loss of generality at this point.
The construction of this tensor product and the properties we use are 
summarized in Appendix \ref{a1}. Here we only mention two
facts. First of all, $H$ is a complete space, i.e., 
a Hilbert space. Moreover, let
$\NN$ denote the set of all sequences $\bnu := (\nu_j)_{j \in
\N}$ in $\N_0$ such that $\nu_j \in N_j$ for every $j \in \N$
and $\sum_{j \in \N} \nu_j < \infty$.
Then the elementary tensors
\[
h_{\bnu} := \bigotimes_{j \in \N} h_{\nu_j,j}
\]
with $\bnu \in \NN$ form an orthonormal basis of the space
$H$.

In the sequel, we use $\scp{\cdot}{\cdot}$ to denote the scalar product 
on the tensor product space $H$.
Of course, the results from Section \ref{s2.0} are applicable
with any set $E$ and any injective linear mapping $\Phi:
H \to \K^E$.
In the present setting it is reasonable, however, 
to require that $\Phi$ respects the tensor 
product structure. Hence we assume in particular that 
\[
E := D^\N
\]
with a set $D \neq \emptyset$.

If we have reproducing kernels $k_j:D\times D\to \K$
for $j \in \N$ such that
\[
K(\bx,\by) := \prod_{j \in \N} k_j(x_j,y_j)
\]
converges for all $\bx, \by \in E$, we write
$$K:= \bigotimes_{j\in\N} k_j.$$ 

We adapt Lemma \ref{l3a} and Lemma \ref{l3b} to the tensor product setting.

\begin{lemma}\label{l5a}
Suppose that $\Phi:H \to \K^{E}$ is linear and injective
and that there exist 
mappings $\Phi_j : \{h_{\nu,j}: \nu \in N_j\} \to \K^D$ such that
\begin{equation}\label{g34}
\forall\, j \in \N: \quad \Phi_j h_{0,j} = 1
\end{equation}
and
\begin{equation}\label{g35}
\forall\, \bnu \in \NN\ \forall\, \by \in E: \quad
\Phi h_\bnu (\by) = \prod_{j \in \N} \Phi_j h_{\nu_j,j}(y_j).
\end{equation}
Furthermore, assume that
$(\Phi(H), \scp{\cdot}{\cdot}_{\Phi})$ is a 
reproducing kernel Hilbert space. Then we have
\begin{equation}\label{g30a}
\forall\,\by \in E: \quad
\sum_{j\in \N}\sum_{\nu \in N_j \setminus \{0\}} 
|\Phi_j h_{\nu,j} (y_j)|^2 < \infty
\end{equation}
and
\begin{equation}\label{g31a}
\forall\,\by \in E\ \forall\, f \in H: \quad
\Phi f (\by) = \sum_{\bnu \in \NN} \scp{f}{h_\bnu} \cdot 
\prod_{j \in \N} \Phi_j h_{\nu_j,j} (y_j)
\end{equation}
with absolute convergence.
Moreover, the reproducing kernel $K$ of this space is 
given by
\begin{equation}\label{g32a}
K = \bigotimes_{j \in \N} k_j,
\end{equation}
where
\[
k_j(x_j,y_j) := 1 + \sum_{\nu \in N_j \setminus \{0\}} 
\Phi_j h_{\nu,j}(x_j) \cdot \overline{\Phi_j h_{\nu,j}(y_j)},
\]
with absolute convergence for all 
$x_j,y_j \in D$. 
\end{lemma}

\begin{proof}
Combine \eqref{g3a} and Lemma \ref{l1} with 
$\beta_{\nu,j} := |\Phi_j h_{\nu,j} (y_j)|^2$
to obtain \eqref{g30a}. In the same way we get
\eqref{g32a} with absolute convergence from \eqref{g22a}.
Finally, \eqref{g31a} with absolute convergence follows
immediately from \eqref{g21}.
\end{proof}

\begin{rem}
Under the assumptions of Lemma \ref{l5a}
every mapping $\Phi_j$ can be extended to a linear injective
mapping $\Phi_j: H_j \to \K^D$ analogously to \eqref{g40}, and
$\Phi_j(H_j)$ is a reproducing kernel Hilbert
space. Moreover, $k_j$
is the reproducing kernel of $\Phi_j(H_j)$, and
$\Phi_j$ is an isometric isomorphism between $H_j$ 
and $H(k_j)$
mapping the unit vector $h_{0,j} \in H_j$ to the function $1\in H(k_j)$.
As noted in Appendix \ref{a1}, this implies that 
the tensor product of the mappings $\Phi_j$
is an isometric 
isomorphism between $H$ and 
$\bigotimes_{j \in \N} H(k_j)$ with unit vectors $\Phi_j h_{0,j}
:= 1$. In particular, $H(K)$ and 
$\bigotimes_{j \in \N} H(k_j)$ 
are canonically isometrically isomorphic.
\end{rem}

In the construction of a mapping $\Phi : H \to \K^E$
we start with injective mappings
$\Phi_j : \{h_{\nu,j}: \nu \in N_j\} \to \K^D$, and we assume
that \eqref{g34} and \eqref{g30a} are satisfied. Due to Lemma
\ref{l1}, the right-hand side in \eqref{g31a} may be used to 
define a linear mapping $\Phi : H \to \K^E$,
satisfying
\eqref{g35}, and Lemma \ref{l3b}
immediately
carries over to the present setting.

\begin{lemma}\label{l3c}
Suppose that \eqref{g34} and \eqref{g30a} are
satisfied and that $\Phi$, defined via \eqref{g31a}, is injective. 
Then $(\Phi(H), \scp{\cdot}{\cdot}_{\Phi})$ is a
reproducing kernel Hilbert space.
\end{lemma}

Next, we adapt Remark \ref{r50}, which deals with $L_2$-spaces,
to the tensor product setting.

\begin{rem}\label{r50a}
Consider a probability measure $\mu_0$ on any 
$\sigma$-algebra on $D$ and the corresponding space
$L_2(D,\mu_0)$.  Then the 
tensor product space $\bigotimes_{j \in \N} L_2(D,\mu_0)$ 
is canonically isometrically isomorphic 
to the space $L_2(E,\mu)$ with respect to the
probability measure $\mu = \mu_0 \times \mu_0 \times \dots$
on the product $\sigma$-algebra on $E$.
Assume that for every $j \in \N$ the space $H_j$ is a subspace of
$L_2(D,\mu_0)$ with a continuous embedding of norm one
and $h_{0,j}=1$. Consequently, $H$ is a subspace of
$\bigotimes_{j \in \N} L_2(D,\mu_0)$ with a continuous embedding of
norm one.

Consider sequences of square-integrable 
functions $\mathfrak{h}_{\nu,j}$ on $D$ with the following properties:
For every $j \in \N$ we have $\mathfrak{h}_{0,j}=1$,  
the corresponding equivalence classes $h_{\nu,j} \in L_2(D,\mu_0)$ 
with $\nu \in N_j$ form an orthonormal basis of $H_j$, and
\[
\forall\,\by \in E: \quad
\sum_{j \in \N} \sum_{\nu \in N_j} |\mathfrak{h}_{\nu,j} (y_j)|^2 
< \infty.
\]
According to Remark \ref{r50} and Lemma \ref{l1} the linear mapping 
$\Phi$ given by 
\[
\Phi f (\by) := 
\sum_{\bnu \in \NN} \scp{f}{h_{\bnu}} \cdot 
\prod_{j \in \N} \mathfrak{h}_{\nu_j,j} (y_j)
\]
for $f \in H$ and $\by \in E$ is injective. 
We apply Lemma \ref{l3c} to conclude that 
$(\Phi(H), \scp{\cdot}{\cdot}_{\Phi})$ is a 
reproducing kernel Hilbert space. 
\end{rem}

\section{Increasing Smoothness and Weights}\label{s3a}

\subsection{The Function Spaces: Abstract Setting}\label{s3.1}

The abstract setting is given by
a separable Hilbert space $(H_0,\scp{\cdot}{\cdot}_0)$ with 
an orthonormal basis $(e_\nu)_{\nu \in \N_0}$ and a family 
$(\alpha_{\nu,j})_{\nu,j \in \N}$ of Fourier weights such that
\begin{equation}\label{g6b}\tag{C1}
\forall\, \nu,j \in \N: \quad
\alpha_{\nu,j} \geq \max\left(\alpha_{\nu,1},\alpha_{1,j}\right)
\end{equation}
and
\begin{equation}\label{g6a}\tag{C2}
\alpha_{1,1} > 1.
\end{equation}

We define
\[
H_j := \{ f \in H_0 \colon 
\sum_{\nu \in \N} \alpha_{\nu,j} \cdot |\scp{f}{e_\nu}_0|^2 
< \infty
\}
\]
and
\[
\scp{f}{g}_j := 
\scp{f}{e_0}_0 \cdot \scp{e_0}{g}_0 +
\sum_{\nu \in \N} \alpha_{\nu,j} \cdot 
\scp{f}{e_\nu}_0 \cdot \scp{e_\nu}{g}_0
\]
for $j \in \N$ and $f,g \in H_j$
to obtain a sequence of Hilbert spaces 
$(H_j,\scp{\cdot}{\cdot}_j)$.
For notational convenience we put $\alpha_{\nu,j} := 1$ for 
$j=0$ and $\nu \in \N_0$ as well as for $j \in \N$ and $\nu=0$.
Clearly
\[
\scp{f}{e_\nu}_j = \alpha_{\nu,j} \cdot \scp{f}{e_\nu}_0
\]
for $\nu,j \in \N_0$ and $f \in H_j$.

We state some basic properties of the spaces $H_j$.
Let $i,j \in \N_0$.
We have a continuous embedding $H_{i} \hookleftarrow H_j$ 
if and only if
\[
\sup_{\nu \in \N} \frac{\alpha_{\nu,i}}{\alpha_{\nu,j}} <
\infty,
\]
and in the case of a continuous embedding its norm
is given by 
\[
\sup_{\nu \in \N_0}
\sqrt{\frac{\alpha_{\nu,i}}{\alpha_{\nu,j}}} \geq 1.
\]
In particular, \eqref{g6b} and \eqref{g6a} imply
$1 \leq \alpha_{\nu,1} \leq \alpha_{\nu,j}$ for $\nu,j \in \N$,
and the latter 
is equivalent to $H_0 \hookleftarrow H_1 \hookleftarrow H_j$ 
with continuous embeddings of norm one
for every $j\geq 1$.
Furthermore, we have a compact embedding $H_{i} \hookleftarrow H_j$
if and only if 
\[
\lim_{\nu \to \infty} \frac{\alpha_{\nu,i}}{\alpha_{\nu,j}} = 0.
\]
Throughout this paper, increasing smoothness is understood 
in this sense, i.e., $H_i \supset H_j$ for $i < j$ with a 
compact embedding.

Let $j \in \N$ and $f \in H_j$. 
The elements $\alpha_{\nu,j}^{-1/2} e_\nu$ with $\nu \in \N_0$ form an
orthonormal basis of the Hilbert space $H_j$.
Let $S_j$ denote the embedding of $H_j$ into $H_0$.  
Since 
$\scp{e_\nu}{e_\mu}_j = \alpha_{\nu,j} \cdot \scp{S_j^* e_\nu}{e_\mu}_j$
for $\nu,\mu \in \N_0$, we obtain
\begin{equation}\label{g400}
S_j^* S_j (f) = \sum_{\nu \in \N_0} \alpha_{\nu,j}^{-1} \cdot
\scp{f}{e_\nu}_0 \cdot e_\nu.
\end{equation}
Consequently, the singular values of $S_j$ are given by
$\alpha_{\nu,j}^{-1/2}$ with $\nu \in \N_0$.

In the abstract setting we consider the tensor product space
\[
H := \bigotimes_{j \in \N} H_j,
\]
based on the choice of the unit vector $e_0$.

\subsection{The Function Spaces: Standard Setting}\label{s3.1a}

Most often, we consider
the following special case of the abstract setting.
This standard setting is given by
\[
H_0 := L_2(D,\mu_0)
\]
for some probability measure $\mu_0$ on a $\sigma$-algebra on any set
$D \neq \emptyset$, by a linear and injective mapping
\[
\Phi_1 : H_1 \to \K^D
\]
that satisfies
\[
\forall \, h \in H_1: \quad \Phi_1(h) \in h
\]
and
\[
\Phi_1 (e_0) = 1,
\]
and by
\[
\Phi_j = \Phi_1|_{H_j}
\]
for $j \geq 2$.

Consequently, the condition \eqref{g3a} reads
\begin{equation}\label{g36}
\forall\,y \in D: \quad
\sum_{\nu\in \N} \alpha_{\nu,j}^{-1} \cdot |\Phi_1 e_\nu (y)|^2 < \infty
\end{equation}
for the space $H_j$, and due to \eqref{g6b}
this condition is most restrictive in the case $j=1$.
Analogously, \eqref{g30a} reads
\begin{equation}\label{g37}
\forall\,\by \in D^\N: \quad
\sum_{\nu, j \in \N} 
\alpha_{\nu,j}^{-1} \cdot |\Phi_1 e_\nu (y_j)|^2 
< \infty
\end{equation}
for the space $H$.
Here it is crucial that the tensor product is based on the choice of 
the unit vectors $e_0$.
Henceforth we typically will not stress this point anymore.
In the standard setting the conditions \eqref{g36} and \eqref{g37}
are necessary and sufficient for $\Phi_1(H_j)$ and
$\Phi (H)$, respectively, to be reproducing kernel Hilbert
spaces, see Remarks \ref{r50} and \ref{r50a}.

Subsequently we identify
$\Phi_1 f$ and $f$ for $f \in H_1$,
$\Phi f$ and $f$ for $f \in H$,
$\Phi_1 (H_j)$ and $H_j$, and
$\Phi (H)$ and $H$, if the respective spaces are reproducing kernel
Hilbert spaces. Furthermore, we do no longer 
distinguish between square-integrable functions on $D$ and elements of 
$H_0$. In this sense, we take $\Phi_1 e_\nu := e_\nu$,
so that, in particular,
\[
e_0:=1. 
\]

In the standard setting 
the space $\bigotimes_{j \in \N} H_0$ is
canonically isometrically isomorphic
to the space $L_2(E,\mu)$, where
$\mu$ denotes the product of the probability measure $\mu_0$ on the 
product $\sigma$-algebra on $E := D^\N$. Obviously, $H$ is a
subspace of $L_2(E,\mu)$ with a continuous embedding of norm one.

\subsection{Examples}\label{s3.2}

In all the examples to be presented below, we consider the standard
setting with a Borel probability measure $\mu_0$ on
an interval $D \subseteq \R$.
We separate the choice of the Hilbert space $H_0$ and its orthonormal 
basis $(e_\nu)_{\nu \in \N_0}$ from the selection of the Fourier
weights $(\alpha_{\nu,j})_{\nu,j \in \N}$. 

See, e.g., \cite{IKPW16a} and the references
therein, for the following example in the context of 
tractability analysis of high-dimensional problems.
For further information about Korobov spaces see, e.g., 
\cite[App.~A.1]{NovWoz08}, 
and about Walsh functions see, e.g., \cite[App.~A]{DP10}. 

\begin{exmp}\label{ex11}
Consider the uniform 
distribution $\mu_0$ on $D:=[0,1]$ together with the 
trigonometric basis $(e_\nu)_{\nu \in \N_0}$, given by 
$e_\nu(x):=\exp(2\pi i (-1)^\nu \lceil \nu/2 \rceil x)$,
or with the Walsh basis $(e_\nu)_{\nu \in \N_0}$, see \cite{Wal23}. 
Since 
$|e_\nu(x)|=1$ for all $\nu \in \N_0$ and $x \in D$, we conclude that
\begin{equation}\label{g9}
\sum_{\nu \in \N} \alpha_{\nu,1}^{-1} < \infty
\end{equation}
is equivalent to
$H_1,H_2,\dots$ being reproducing kernel Hilbert spaces.
Furthermore, $H$ is a 
reproducing kernel Hilbert space if
and only if
\[
\sum_{\nu,j \in \N} \alpha_{\nu,j}^{-1} < \infty.
\]
If the spaces $H_j$ stem from the trigonometric basis, then they are 
known as Korobov spaces. If they stem from the Walsh basis, then they 
are often called Walsh spaces. 
\end{exmp}

For the next example see, for instance, 
\cite{HHY04} 
and the references therein. 

\begin{exmp}\label{eha}
Consider the uniform distribution $\mu_0$ on $D:=[0,1]$ together with the 
$L_2$-normalized Haar basis $(e_\nu)_{\nu \in \N_0}$. 
Put $I_\ell := \{2^\ell,\dots,2^{\ell+1}-1\}$ for $\ell \in \N_0$,
and assume that
\begin{align}\label{condpower}
\alpha_{2^\ell,j} = \dots = \alpha_{2^{\ell+1}-1,j}
\end{align}
for $\ell \in \N_0$ and $j \in \N$.
Since
\[
\sum_{\nu \in I_\ell} |e_\nu (x)|^2 = 2^{\ell} 
\]
for all $x \in D$ and $\ell \in \N_0$, the conclusions from
Example \ref{ex11} are also valid in the present case.
Since the Haar functions $e_\nu$ as well as the Walsh functions 
$e_\nu$ from Example \ref{ex11} with $\nu\in I_\ell$ are an orthonormal 
basis of the same finite-dimensional subspace of $L_2([0,1], \mu_0)$, 
condition  \eqref{condpower} ensures that in both cases we obtain 
the same sequence of Hilbert spaces $H_j$.
\end{exmp}

\begin{exmp}\label{ex13}
Consider the uniform 
distribution $\mu_0$ on $D:=[-1,1]$ together with the $L_2$-normalized
Legendre polynomials $e_\nu$. Here we have
\[
e_\nu(1) = \|e_\nu\|_\infty := 
\sup_{x \in D} |e_\nu(x)| = \sqrt{2\nu+1} \asymp \max(\nu^{1/2},1),
\]
see, e.g., \cite[Ex.~2.20]{MR0249885}.
It follows that
$H_1,H_2,\dots$ are reproducing kernel Hilbert spaces
if and only if
\[
\sum_{\nu \in \N} \alpha_{\nu,1}^{-1} \cdot \nu < \infty,
\]
while $H$ is a reproducing kernel Hilbert space 
if and only if
\[
\sum_{\nu,j \in \N} \alpha_{\nu,j}^{-1} \cdot \nu < \infty.
\]
\end{exmp}

\begin{exmp}\label{s3:e4}
Now we consider a generalization of Example \ref{ex13}. 
Let $\mu_0$ be defined by the Lebesgue density 
$x \mapsto c^{(\alpha,\beta)}\cdot 
(1-x)^\alpha (1+x)^\beta$ on $D := [-1,1]$
for some $\alpha,\beta > -1/2$,
where
\begin{align*}
c^{(\alpha,\beta)}
:=\frac{\alpha+\beta+1}{2^{\alpha+\beta+1}}\cdot 
\binom{\alpha+\beta}{\alpha}.
\end{align*}
The orthogonal polynomials associated 
to this weight function are the Jacobi polynomials 
$P_\nu^{(\alpha,\beta)}$, usually normalized such that 
$P_\nu^{(\alpha,\beta)}(1)= \binom{\nu+\alpha}{\nu}$, 
see, e.g., \cite[Eqn.~(4.1.1)]{S75}.
The special case $\alpha=\beta=0$ yields the Legendre polynomials.
The $L_2$-normalized version is
\[ 
 e_\nu := c_\nu^{(\alpha,\beta)} \cdot P_\nu^{(\alpha,\beta)} 
\]
with
\begin{align*}
c_\nu^{(\alpha,\beta)} &:=
(c^{(\alpha,\beta)})^{-1/2}\cdot
\left( \frac{(2\nu+\alpha+\beta+1) \cdot \Gamma(\nu+1)  
\cdot \Gamma(\nu+\alpha+\beta+1)}{2^{\alpha+\beta+1} \cdot 
\Gamma(\nu+\alpha+1) \cdot \Gamma(\nu+\beta+1)} \right)^{1/2}\\
&\asymp \max(\nu^{1/2},1),
\end{align*}
see, e.g., \cite[Eqn.~(4.3.3)]{S75}.
The Jacobi polynomials 
$P_\nu^{(\alpha,\beta)}$ attain their supremum norm in 
$-1$ or in $+1$ with 
\[
 \| P_\nu^{(\alpha,\beta)}
 \|_\infty= \binom{\nu+\max (\alpha,\beta)}{\nu} 
 \asymp \max(\nu^{\max (\alpha,\beta)},1),
\]
see, e.g., \cite[Thm.~7.32.1]{S75}. 
Altogether we obtain 
\[
\max(|e_\nu(1)|,|e_\nu(-1)|)
=\sup_{x \in D} |e_\nu(x)| \asymp \max(\nu^{\sigma},1)
\]
with $$ \sigma:=\max(\alpha,\beta)+\frac12 > 0.$$
It follows that
$H_1,H_2,\dots$ are reproducing kernel Hilbert spaces
if and only if
\[
\sum_{\nu \in \N} \alpha_{\nu,1}^{-1} \cdot \nu^{2 \sigma} < \infty,
\]
while $H$ is a reproducing kernel Hilbert space 
if and only if
\[
\sum_{\nu,j \in \N} \alpha_{\nu,j}^{-1} \cdot \nu^{2 \sigma} < \infty.
\]
\end{exmp}

\begin{rem}
In the Examples \ref{ex11}--\ref{s3:e4} the summability of 
$(\alpha_{\nu,j}^{-1} \cdot \nu^\sigma)_{\nu,j \in \N}$ for some
$\sigma \geq 0$
determines whether $H$ is a reproducing kernel Hilbert space. 
According to Lemma \ref{l4} this summability already follows from 
the summability of 
$(\alpha^{-1}_{\nu,1} \cdot \nu^\sigma)_{\nu \in\N}$ and 
$(\alpha_{1,j}^{-1})_{j\in\N}$, if
\begin{equation}\label{g5}
\liminf_{\nu,j \to \infty} 
\frac{\ln(\alpha_{\nu,j})}{\ln(\nu) \cdot \ln (j)} > 0.
\end{equation}
\end{rem}

Next, we turn to two important classes of Fourier
weights. At first we introduce some notation.
The decay of any sequence $x=(x_i)_{i \in \N}$ of positive reals 
is defined by
\[
\decay (x) := 
\sup \Bigl\{ \tau > 0 : \sum_{i \in \N} x_i^{1/\tau} < \infty \Bigr\}
\]
with the convention that $\sup \emptyset := 0$,
see
\cite[p.~311]{Was12}.
As a well-known fact
\[
\decay (x) = 
\liminf_{i \to \infty} \frac{\ln(x_i^{-1})}{\ln(i)}
\]
if the decay or the limes inferior is positive,
which follows, e.g., from Lemma~\ref{l5} in Appendix B.

\begin{exmp}\label{ex1}
We consider
\[
\alpha_{\nu,j} := a_\nu^{r_j},
\]
where
\begin{equation}\label{g15}
\forall\, j \in \N:\quad 0 < r_1 \leq r_j
\end{equation}
\[
\]
as well as
\[
\forall\, \nu \in \N: \quad 1 < a_1 \leq a_\nu
\]
and
\begin{equation}\label{g60}
a_\nu \asymp \nu.
\end{equation}

Put $r_0:=0$. For $j \in \N_0$
the space $H_{j+1}$ is continuously embedded into $H_j$ 
(with norm one) if and only if $r_j \leq r_{j+1}$,
and in this case 
$r_j < r_{j+1}$ is equivalent to the compactness of
this embedding.

Obviously \eqref{g6b} and \eqref{g6a} hold true, and \eqref{g5}
is equivalent to
\[
\rho > 0
\]
for
\[
\rho := \liminf_{j \to \infty} \frac{r_j}{\ln(j)}.
\]

Note that
\[
\decay ((\alpha^{-1}_{\nu,1})_{\nu \in \N}) =
r_1
\]
and
\[
\decay ((\alpha^{-1}_{1,j})_{j \in \N}) = \rho \cdot \ln (a_1).
\]
Let $\sigma \geq 0$.
Observe that
\begin{equation}\label{g200}
\sum_{\nu \in \N} \alpha_{\nu,1}^{-1} \cdot \nu^\sigma < \infty
\end{equation}
is actually equivalent to
\[
r_1 > \sigma + 1,
\]
while
\[
r_1 > \sigma + 1
 \ \wedge \ \rho > \frac{1}{\ln(a_1)}
\]
is a sufficient condition for 
\begin{equation}\label{g201}
\sum_{\nu,j \in \N} \alpha_{\nu,j}^{-1} \cdot \nu^\sigma < \infty
\end{equation}
to hold. A necessary condition also permits $\rho = 1/\ln(a_1)$.
\end{exmp}

\begin{rem}
The exponents $r_j$ in Example~\ref{ex1} may be regarded as 
smoothness parameters.
To illustrate this point, we first consider the complex $L_2$-space and the 
complex exponentials according to Example \ref{ex11}.
Up to equivalence of norms, the Korobov spaces $H_j$ with
parameters $r_j$
may be defined by any choice of $a_\nu>0$ such that
\eqref{g60} is satisfied. 
Specifically
\begin{equation}\label{g1}
a_\nu := 2\pi \lfloor (\nu+1) / 2 \rfloor
\end{equation}
is considered in, e.g., \cite{PapWoz10,Sie14} and
\begin{equation}\label{g2}
a_\nu := 1 + \lfloor (\nu+1)/2 \rfloor
\end{equation}
is considered in, e.g., \cite{DunGri16}. 
Observe that the index set $\Z$ instead of $\N_0$
is considered in \cite{DunGri16,PapWoz10,Sie14}. Furthermore,
the parameters $2r_j$ instead of $r_j$ are used in
\cite{PapWoz10,Sie14}.
See \cite{KSU17} for a generalization of this type of Fourier
weights, which involves an additional fine parameter.

Secondly, we consider the smoothness spaces 
based on Legendre polynomials and, more general, on 
Jacobi polynomials in Examples \ref{ex13} and \ref{s3:e4}, 
which are related to weighted Sobolev spaces. Such spaces 
were considered in, e.g., \cite{Nic00}. We discuss one special case 
where the relation can be directly explained. Corollary 2.6 and 
Theorem 2.7 from \cite{Nic00} show that, if $r_j$ is an even integer 
and $\alpha = \beta >-1/2$, then the space $H_j$ with respect to 
the Jacobi polynomials $P_\nu^{(\alpha,\alpha)}$ can be identified 
(with equivalent norms) with the Sobolev space of all functions on 
$(-1,1)$ with weak derivatives up to order $r_j/2$ in the weighted 
$L_2$-space of functions on $(-1,1)$ with respect to the weight 
function $\varrho_{\alpha,r_j}(x) = (1-x^2)^{\alpha+r_j/2}$.

More formally, let $L_2(\varrho_{\alpha,r_j})$ be the Hilbert space 
of all functions $f:(-1,1) \to \R$ with 
\[
  \|f\|_{L_2(\varrho_{\alpha,r_j})}^2
  =
 \int_{-1}^1 |f(x)|^2 \varrho_{\alpha,r_j}(x) \, d x < \infty.
\]
Let $W^{r_j/2} (\varrho_{\alpha,r_j})$ be the Hilbert space
of all functions $f$ on $(-1,1)$ with weak derivatives up to order 
$r_j/2$ in $L_2(\varrho_{\alpha,r_j})$ with norm
given by
\[
 \left( \sum_{k=0}^{r_j/2} 
\|f^{(k)}\|_{L_2(\varrho_{\alpha,r_j})}^2 \right)^{1/2}.
\]
Then 
\[ 
 H_j = W^{r_j/2} (\varrho_{\alpha,r_j})
\] 
with equivalent norms.
\end{rem}

\begin{exmp}\label{ex2a}
Choose $a > 1$ and consider
\[
\alpha_{\nu,j} := a^{r_j \cdot \nu^{b_j}}
\]
with \eqref{g15} being satisfied and with
\[
\forall\, j \in \N: \quad
0 < b_1 \leq b_j.
\]
See, e.g., \cite{IKPW16a} and the references therein for this
type of Fourier weights.

Put $r_0:=0$ as previously. For $j \in \N_0$
we have a compact embedding of $H_{j+1}$ into $H_j$
if and only if $b_j < b_{j+1}$ or $b_j = b_{j+1}$ and $r_j <
r_{j+1}$. Furthermore, we have a continuous, non-compact embedding
only in the trivial case $b_j=b_{j+1}$ and $r_j=r_{j+1}$.

Obviously \eqref{g6b} and \eqref{g6a} hold true, and \eqref{g5}
follows from
\[
\rho > 0,
\]
where $\rho$ is defined as in Example \ref{ex1}.
In contrast to Example \ref{ex1},
we now have (sub-)exponentially growing Fourier weights for every space
$H_j$ with $j \in \N$. In particular,
\[
\decay ((\alpha^{-1}_{\nu,1})_{\nu \in \N}) = \infty,
\]
while
\[
\decay ((\alpha^{-1}_{1,j})_{j \in \N}) = \rho \cdot \ln (a).
\]
Hence \eqref{g200} is satisfied for every $\sigma \geq 0$.
A sufficient condition for \eqref{g201} to hold
is
\[
\rho > \frac{1}{\ln(a)}.
\]
Again a necessary condition also permits equality.
\end{exmp}

\subsection{The Embeddings: Abstract Setting}\label{secembed}

Consider the abstract setting.
Let
\[
\gamma_j := \sup_{\nu \in \N} \frac{\alpha_{\nu,1}}{\alpha_{\nu,j}}
\]
for $j \in \N$, and observe that $0< \gamma_j \leq 1$ due to 
\eqref{g6b} and \eqref{g6a}.

For the first kind of embedding 
we use the sequence $(\alpha_{\nu,1})_{\nu \in \N}$ of Fourier
weights of the space $(H_1,\scp{\cdot}{\cdot}_1)$ 
and the sequence $(\gamma_j)_{j \in \N}$ of positive weights 
to construct a new sequence of Hilbert spaces
$(G_j,\scp{\cdot}{\cdot}_{G_j})$ in the following
way. We take
\[
G_j := H_1
\]
and
\[
\scp{f}{g}_{G_j} := 
\scp{f}{e_0}_0 \cdot \scp{e_0}{g}_0 + 
\frac{1}{\gamma_j} \cdot
\sum_{\nu \in \N} 
\alpha_{\nu,1} \cdot \scp{f}{e_\nu}_0 \cdot \scp{e_\nu}{g}_0
\]
for $j \in \N$ and $f,g \in H_1$.
Of course, this is a particular case of the construction
of the spaces $(H_j,\scp{\cdot}{\cdot}_j)$, where the Fourier
weights are now of the form $\alpha_{\nu,j} :=
\alpha_{\nu,1}/\gamma_j$.
In addition to $H$ we consider the tensor product 
space
\[
G := \bigotimes_{j \in \N} G_j
\]
with the corresponding scalar product.

\begin{rem}
The embeddings $G_j \hookleftarrow G_{j+1}$ and 
$G_j \hookrightarrow G_{j+1}$ are 
continuous with norms 
$\max(1,\sqrt{\gamma_{j+1}/\gamma_j})$ and
$\max(1,\sqrt{\gamma_j/\gamma_{j+1}})$, respectively. 
In particular, we have equivalence of the norms
on all spaces
$(G_j,\scp{\cdot}{\cdot}_{G_j})$, which is
in sharp contrast to spaces of increasing smoothness,
where we have compact embeddings 
$H_j \hookleftarrow H_{j+1}$.
\end{rem}

For the second kind of embedding we take 
\[
F_j := \spann \{e_0,e_1\}
\]
as well as
\[
\scp{f}{g}_{F_j} :=
\scp{f}{e_0}_0 \cdot \scp{e_0}{g}_0 + 
\alpha_{1,j} \cdot
\scp{f}{e_1}_0 \cdot \scp{e_1}{g}_0
\]
for $j \in \N$ and $f,g \in F_1$,
and we consider the tensor product space
\[
F := \bigotimes_{j \in\N} F_j
\]
with the corresponding scalar product.

Our analysis is based on the following simple observation.

\begin{theo}\label{t0}
In the abstract setting we have
\[
F \hookrightarrow H \hookrightarrow G
\]
with embeddings of norm one.
\end{theo}

\begin{proof}
The norm of the embeddings $H_j \hookrightarrow G_j$
and $F_j \hookrightarrow H_j$ is one. 
\end{proof}

The spaces $G$ and $F$
are so-called weighted tensor product spaces, which have
been intensively studied.  
Weighted tensor product spaces of functions depending on 
finitely many variables were introduced in \cite{SW98} for the 
analysis of tractability of multivariate 
problems; for further results and references see, e.g., 
\cite{DKS13, NovWoz08, NovWoz10, NovWoz12}.
Weighted tensor spaces of functions depending on infinitely many 
variables were first considered in
\cite{HW01}. 
The structure of the spaces is analyzed in \cite{GMR14} and a 
survey of recent results on infinite-dimensional  
integration on such spaces can be found in \cite{GneEtAl16}. 

In the present setting the weighted tensor products are
based on the spaces $(H_1,\scp{\cdot}{\cdot}_1)$ and 
$(\spann \{e_0,e_1\},\scp{\cdot}{\cdot}_1)$ and on the weights
$\gamma_j$ and $\alpha_{1,1}/\alpha_{1,j}$, respectively.
Theorem \ref{t0} allows to transfer
results from weighted tensor product spaces to
tensor products of spaces of increasing
smoothness and vice versa.

The results that will be derived in the subsequent sections
depend on the family $(\alpha_{\nu,j})_{\nu,j \in \N}$
of Fourier weights via the decays of the
sequences $(\alpha_{\nu,1}^{-1})_{\nu \in \N}$,
$(\alpha_{1,j}^{-1})_{j \in \N}$, and
$(\gamma_j)_{j \in \N}$.

\begin{exmp}\label{exmp1}
In the situation of Example \ref{ex1} 
we have
\[
\gamma_j = \sup_{\nu \in \N} a_\nu^{r_1-r_j}  = a_1^{r_1-r_j},
\]
and therefore
\[
\decay\left((\gamma_j)_{j \in \N}\right) = \rho \cdot \ln(a_1) =
\decay\left((\alpha_{1,j}^{-1})_{j \in \N}\right). 
\]
Analogously, in the situation of Example \ref{ex2a}, 
\[
\gamma_j = 
\sup_{\nu \in \N} a^{r_1\cdot \nu ^{b_1} -r_j \cdot \nu^{b_j}}  
= a^{r_1-r_j},
\]
and therefore
\[
\decay\left((\gamma_j)_{j \in \N}\right) = \rho \cdot \ln(a) =
\decay\left((\alpha_{1,j}^{-1})_{j \in \N}\right). 
\]
\end{exmp}

\begin{rem}\label{r70}
The family $(\alpha_{\nu,1}/\gamma_j)_{\nu,j \in \N}$ of
Fourier weights satisfies \eqref{g6b} and \eqref{g6a} as well.
However, if also \eqref{g5} holds true for
$(\alpha_{\nu,j})_{\nu,j\in\N}$, we 
do not necessarily have this property for 
$(\alpha_{\nu,1}/\gamma_j)_{\nu,j \in \N}$.
Nevertheless, it is easy to see that the conclusion of Lemma 
\ref{l4} still holds true for the latter family of Fourier weights.
\end{rem}

\subsection{The Embeddings: Standard Setting}\label{sec35}

Now we turn to the standard setting, and we assume that
$G$ is a reproducing kernel Hilbert space 
(in the sense of the study from Section \ref{s2.2}).
It follows that each of the spaces $H_j$, $G_j$, or $F_j$
is a Hilbert space with a reproducing kernel
of the form $1+m$, where $m$ is a reproducing kernel as well
and $H(1) \cap H(m) = \{0\}$.

Consider any reproducing kernel $m$ on $D \times D$. 
If there exists a point $a \in D$ such that $m(a,a)=0$,
then $m$ is called an anchored kernel with anchor $a$.
The latter is equivalent to $f(a)=0$
for every $f \in H(m)$. Next, consider the reproducing kernel
$1+m$, and suppose that $H(1) \cap H(m) = \{0\}$. Then $m$ is an anchored
kernel with anchor $a$ if and only if 
the orthogonal projection onto the subspace $H(1)$ of constant functions 
in $H(1+m)$ is given by $f \mapsto f(a)$.
An anchored kernel induces an anchored function space decomposition on
$\otimes^d_{j=1} H(1+\gamma_j m)$ with $d\in \N$, see \cite{KSWW10}, and on
$\otimes_{j\in \N} H(1 + \gamma_j m)$, see \cite{GMR14}.
Individual components of this decomposition can be evaluated efficiently 
using function values only, see again \cite{KSWW10}.

We stress that for each of the spaces $H_j$, $G_j$, or $F_j$
the respective kernel $m$ is not necessarily
anchored. 
Actually, all the spaces $H_j$ and $G_j$ that we obtain in the 
Examples \ref{ex11} to \ref{s3:e4} do not have a reproducing kernel 
$1+m$ with an anchored kernel $m$. This is easily verified: Since 
$H(m)$ is the orthogonal complement of $H(1)$ in $H(1+m)$, we have that 
$e_1, e_2 \in H(m)$. If $m(a,a)=0$ for some $a\in D$, then necessarily 
$e_1(a)= 0=e_2(a)$. But in the Examples \ref{ex11} and \ref{eha} we 
have $|e_1(x)|=1=|e_2(x)|$ for all $x\in D$.
In Example \ref{s3:e4} (and thus also in Example \ref{ex13}, which is a 
special case of the former example)
the only zero of $e_1$ is 
$a := (\beta -\alpha)/(\alpha +\beta +2)$, and it is easily checked that
$e_2(a) \neq 0$.
Furthermore, we have that the kernel $m(x,y) = \alpha_{1,j}^{-1}\, 
e_1(x)\overline{e_1(y)}$  corresponding to $F_j$ is not anchored in the 
Examples \ref{ex11} and \ref{eha} and  anchored in 
$a := (\beta -\alpha)/(\alpha +\beta +2)$ in Example \ref{s3:e4} and, 
consequently, in $a:=0$ in Example \ref{ex13}.

We establish, however, relations between the 
spaces $H_j$, $G_j$ and $F_j$ 
and spaces with anchored kernels via suitable embeddings.

To this end,
we fix a point $a \in D$,
and for $j \in \N$ and $c > 0$ we define
\[
\gjc := G_j = H_1
\]
and
\[
\scp{f}{g}_{\gjc} := 
f(a) \cdot \overline{g(a)} + 
\frac{1}{c \gamma_j} \cdot
\sum_{\nu \in \N} 
\alpha_{\nu,1} \cdot \scp{f}{e_\nu}_0 \cdot \scp{e_\nu}{g}_0,
\]
where $f,g \in H_1$.

In the sequel we employ results from \cite{GneEtAl16},
which have been formulated for reproducing kernel Hilbert spaces
of real-valued functions. These results may be extended to
complex-valued functions in a canonical way and are thus applicable 
in the present setting. 

\begin{lemma}\label{lw1}
For all $j \in \N$ and $c>0$ the space
$(\gjc, \scp{\cdot}{\cdot}_{\gjc})$ is a reproducing kernel Hilbert
space of functions with domain $D$, and its norm is equivalent to 
$\|\cdot\|_{G_1}$. 
Moreover, there exists a (uniquely defined) reproducing
kernel $m$ on $D \times D$ such that $1+ c\gamma_j \cdot m$ is the 
reproducing kernel of 
$(\gjc, \scp{\cdot}{\cdot}_{\gjc})$ 
for all $j \in \N$ and $c>0$, and
\begin{equation}\label{gw1}
m(a,a) = 0.
\end{equation}
\end{lemma}

\begin{proof}
Put
\[
\|f\|_{1,\na} := |\scp{f}{e_0}_0|
\]
and
\[
\|f\|_{1,\nb} := |f(a)|
\]
as well as
\[
\|f\|_{2,\na}^2 := \|f\|_{2,\nb}^2 := 
\sum_{\nu \in \N} 
\alpha_{\nu,1} \cdot |\scp{f}{e_\nu}_0|^2
\]
for $f \in H_1$.
According to \cite[Rem.~2.1]{GneEtAl16},
the vector space $H_1$ together with the seminorms
$\|\cdot\|_{1,\na}$ and $\|\cdot\|_{2,\na}$ satisfies
the conditions \cite[(A1)--(A3)]{GneEtAl16}. 
The same holds true for the seminorms
$\|\cdot\|_{1,\nb}$ and $\|\cdot\|_{2,\nb}$, see
\cite[Rem.~2.5]{GneEtAl16}.

By definition of \cite[(A3)]{GneEtAl16} this ensures, in particular,
that $\scp{\cdot}{\cdot}_{\gjc}$ is a scalar product on $\gjc$ that
turns the latter space into a reproducing kernel Hilbert space.
The closed graph theorem yields the equivalence of norms as
claimed.

Let $m$ denote the reproducing kernel of $\{f \in H_1 : f(a)=0\}$ in
$(\gjc, \scp{\cdot}{\cdot}_{\gjc})$ in the particular case $c \gamma_j=1$.
By definition, we have \eqref{gw1}, and 
\cite[Lem.~2.1, Rem.~2.2]{GneEtAl16} imply that the reproducing
kernel of 
$(\gjc, \scp{\cdot}{\cdot}_{\gjc})$ is given by $1 + c\gamma_j
\cdot m$ for all $j \in \N$ and $c>0$. 
\end{proof}

We stress the following important differences between the spaces
$(G_j,\scp{\cdot}{\cdot}_{G_j})$ and
$(\gjc, \scp{\cdot}{\cdot}_{\gjc})$.
In the latter case the orthogonal projection onto the space of
constant functions is easy to compute, but $(e_\nu)_{\nu \in \N_0}$ 
is an orthogonal system only in the trivial case that
the two scalar products of $G_j$ and $G^1_j$ coincide.

\begin{lemma}\label{lw2}
There exists a constant $0<c_0<1$ such that
\begin{equation}\label{gw2}
(1+c_0^{-1}\gamma_j)^{-1/2} \cdot \|f\|_{G_j^{c_0^{-1}}}
\leq \|f\|_{G_j} \leq
(1+\gamma_j)^{1/2} \cdot \|f\|_{G_j^{c_0}}
\end{equation}
and
\[
\|f\|_0 \leq (1+c_0^{-2}\gamma_j) \cdot \|f\|_{G_j^{c_0^{-1}}} 
\]
for all $j \in \N$ and $f \in H_1$.
\end{lemma}

\begin{proof}
According to the first paragraph of the proof of Lemma \ref{lw1}
we are in the situation from \cite{GneEtAl16}.
The inequality \eqref{gw2} follows directly from \cite[Thm.~2.1]{GneEtAl16}
and Lemma \ref{lw1}. 

Analogously, the norm of the embedding 
$G_j^{c_0^{-1}} \hookrightarrow \tilde{G}_j$ is bounded from above by
$(1+c_0^{-2}\gamma_j)^{1/2}$, where $\tilde{G}_j$ is defined as
$G_j$, however with new weights $c_0^{-2} \cdot \gamma_j$ instead
of $\gamma_j$.
Furthermore,
the norm of the embedding $\tilde{G}_j \hookrightarrow H_0$
is given by 
$\max\left( \sqrt{c_0^{-2} \gamma_j/\alpha_{1,1}},1\right)$, 
cf.\ Section \ref{s3.1}. 
It remains to observe that
\[
\max\left( \sqrt{c_0^{-2} \gamma_j/\alpha_{1,1}},1\right) \cdot
(1+c_0^{-2}\gamma_j)^{1/2}
\leq 
1+c_0^{-2}\gamma_j. 
\qedhere
\]
\end{proof}

Condition~\eqref{g37} for the space $G$ reads
\begin{equation*}
\forall\,\by \in D^\N: \quad
\sum_{\nu, j \in \N} 
(\alpha_{\nu,1}/\gamma_j)^{-1} \cdot |e_\nu (y_j)|^2 
< \infty.
\end{equation*}
Considering $\nu=1$ and some $y\in D$ such that $e_1(y)\neq 0$ yields
\begin{align}\label{sumfinite}
\sum_{j\in\N}\gamma_j<\infty.
\end{align}

For $c > 0$ we define
\begin{align*}
G^c
:=\bigotimes_{j\in\N} G_j^c.
\end{align*}
Note that different values of $c$ may lead to different spaces and not 
just to different norms,
see \cite{MR3325681},
and the spaces do not necessarily fit into the setting of 
Section~\ref{s3.1}.

\begin{lemma}\label{emb1dimg}
For every $c>0$ the space $G^c$ is a reproducing kernel Hilbert space of 
functions with domain $D^\N$ and reproducing kernel given by 
$\bigotimes_{j\in\N}(1+c\gamma_j\cdot m)$ with $m$ according to 
Lemma~\ref{lw1}. Furthermore, there exists a constant $0<c_0<1$ such 
that we have continuous embeddings
\begin{align*}
G^{c_0}
\hookrightarrow
G
\hookrightarrow
G^{c_0^{-1}}
\hookrightarrow
\bigotimes_{j\in\N} H_0.
\end{align*}
\end{lemma}

\begin{proof}
According to the first paragraph of the proof of Lemma~\ref{lw1}
we are in the situation from \cite{GneEtAl16}.

Combing \eqref{sumfinite} with \cite[Thm.~2.3]{GneEtAl16} yields the 
first claim. The second claim follows directly from Lemma~\ref{lw2} 
and \eqref{sumfinite}.
\end{proof}

We proceed in the same way for the space $F$.
For $j \in \N$ and $c > 0$ we define
\[
\fjc := F_j = F_1
\]
and
\[
\scp{f}{g}_{\fjc} := 
f(a) \cdot \overline{g(a)} + 
\frac{\alpha_{1,j}}{c} \cdot
\scp{f}{e_1}_0 \cdot \scp{e_1}{g}_0,
\]
where $f,g \in F_1$.

\begin{lemma}\label{lem1234}
For all $j \in \N$ and $c>0$ the space
$(\fjc, \scp{\cdot}{\cdot}_{\fjc})$ is a reproducing kernel Hilbert
space of functions with domain $D$, and its norm is equivalent to 
$\|\cdot\|_{F_1}$. 
Moreover, there exists a (uniquely defined) reproducing
kernel $\ell$ on $D \times D$ such that 
$1+ c\alpha_{1,j}^{-1} \cdot \ell$ is the 
reproducing kernel of 
$(\fjc, \scp{\cdot}{\cdot}_{\fjc})$ 
for all $j \in \N$ and $c>0$, and
\[
\ell(a,a) = 0.
\]
\end{lemma}

\begin{lemma}\label{l200}
There exists a constant $0<c_0<1$ such that
\[
(1+c_0^{-1}\alpha_{1,j}^{-1})^{-1/2} \cdot \|f\|_{F_j^{c_0^{-1}}}
\leq \|f\|_{F_j} \leq
(1+\alpha_{1,j}^{-1})^{1/2} \cdot \|f\|_{F_j^{c_0}}
\]
and
\[
\|f\|_0 \leq (1+c_0^{-2}\alpha_{1,j}^{-1}) \cdot \|f\|_{F_j^{c_0^{-1}}} 
\]
for all $j \in \N$ and $f \in F_1$.
\end{lemma}

Observe that $F$ is a reproducing kernel Hilbert space.
Furthermore, from \eqref{sumfinite} we get
\[
\sum_{j\in\N}\alpha_{1,j}^{-1}<\infty.
\]

For $c > 0$ we define
\begin{align*}
F^c
:=\bigotimes_{j\in\N} F_j^c.
\end{align*}

\begin{lemma}\label{emb1dimf}
For every $c>0$ the space $F^c$ is a reproducing kernel Hilbert space 
of functions with domain $D^\N$ and reproducing kernel given by 
$\bigotimes_{j\in\N}(1+c\alpha_{1,j}^{-1}\cdot \ell)$ with $\ell$ 
according to Lemma~\ref{lem1234}. Furthermore, there exists a constant 
$0<c_0<1$ such that we have continuous embeddings
\begin{align*}
F^{c_0}
\hookrightarrow
F
\hookrightarrow
F^{c_0^{-1}}
\hookrightarrow
\bigotimes_{j\in\N} H_0.
\end{align*}
\end{lemma}

Combining Theorem \ref{t0}, Lemma~\ref{emb1dimg}, and
Lemma~\ref{emb1dimf} yields the following result.

\begin{theo}\label{t1}
Consider the standard setting, and assume that $G$ is a reproducing
kernel Hilbert space.
Then there exists a constant $0 < c_0 <1$ with the following
properties. We  have continuous 
embeddings
\begin{center}
\begin{tikzcd}
F^{c_0}\arrow[hook]{d}
&
&G^{c_0}\arrow[hook]{d}\\
F\arrow[hook]{d}\arrow[hook]{r}&H\arrow[hook]{r}&G\arrow[hook]{d}\\
F^{c_0^{-1}}\arrow[hook]{d}&&G^{c_0^{-1}}\arrow[hook]{d}
\\
\bigotimes_{j\in\N} H_0&&\bigotimes_{j\in\N} H_0,
\end{tikzcd}
\end{center}
and $F^{c_0^{-1}}$ as well as $G^{c_0^{-1}}$
are reproducing kernel Hilbert spaces.
\end{theo}

\section{Infinite-Dimensional Approximation and Integration}\label{s4}

Consider a bounded linear operator 
$S\colon \mathcal{X}\to \mathcal{Z}$ 
between two $\K$-Hilbert spaces as well as a non-decreasing
sequence $\mathcal{A} = (\mathcal{A}_n)_{n \in \N}$ of sets
$\mathcal{A}_n$ of bounded linear operators
between $\mathcal{X}$ and $\mathcal{Z}$.
We study the corresponding $n$-th minimal worst case error 
\begin{align*}
\err_n(S,\mathcal{A}) :=
\inf_{A \in \mathcal{A}_n}
\sup
\{ \|S(f)-A(f)\|_{\mathcal{Z}} \colon
f\in \mathcal{X},\ \|f\|_{\mathcal{X}}\leq 1\};
\end{align*}
more precisely, we determine
\[
\dec (S,\mathcal{A}) := 
\decay \left( 
\left(\err_n(S,\mathcal{A})\right)_{n \in \N} \right).
\]
To stress the dependence on $\mathcal{X}$ we often write
$\err_n(\mathcal{X},S,\mathcal{A})$ and 
$\dec(\mathcal{X},S,\mathcal{A})$.
Observe that
\[
\decay(x) = 
\sup\{\tau > 0 : \sup_{i \in \N} x_i \cdot i^\tau < \infty\}
\]
for any non-increasing sequence $x = (x_i)_{i \in \N}$ of positive
reals, cf. \cite[p.~311]{Was12}, and note that lower bounds for 
$\dec (S,\mathcal{A})$ correspond to upper bounds for the $n$-th
minimal errors $\err_n(S,\mathcal{A})$ and vice versa.

For the approximation problem we have $\mathcal{X} \subseteq
\mathcal{Z}$ with a continuous embedding, and $S = \App$
is the corresponding embedding operator
\[
\App\colon 
\mathcal{X} \hookrightarrow \mathcal{Z}.
\]
For the integration problem we have $\mathcal{Z} = \K$, 
and $S=\Int$ is a bounded linear functional 
\[
\Int\colon 
\mathcal{X} \to \K,
\]
which is defined by means of integration with respect to a probability
measure.

Theorems \ref{t0} and \ref{t1} allow us to derive
results for linear problems, like approximation and integration,
on the tensor product $H$ of spaces of increasing smoothness
from known results for the weighted tensor product 
spaces $F$ and $G$ or $F^{c_0}$ and $G^{c_0^{-1}}$, where the
latter pair of spaces is, additionally, based on anchored kernels.
Under the corresponding assumptions we have
\begin{equation}\label{g300}
\dec(G,S,\mathcal{A}) \leq 
\dec(H,S,\mathcal{A}) \leq 
\dec(F,S,\mathcal{A})
\end{equation}
and 
\begin{equation}\label{g301}
\dec(G^{c_0^{-1}},S,\mathcal{A}) \leq 
\dec(H,S,\mathcal{A}) \leq 
\dec(F^{c_0},S,\mathcal{A}).
\end{equation}
Furthermore, $H_1$ can be isometrically
embedded into $H$ via 
$f\mapsto f\otimes \big( \bigotimes_{n\ge 2}\, e_0 \big)$. 
If we identify $H_1$ with its image under this embedding, we may 
consider a non-decreasing sequence 
$\bb = \big( \bb_n \big)_{n\in \N}$ 
of sets $\bb_n$ of bounded
linear operators between $H_1$ and $\mathcal{Z}$ that satisfies
$A|_{H_1} \in \bb_n$ for every $A \in \mathcal{A}_n$ and
every $n  \in \N$.
Then we have 
\begin{equation}\label{g302}
\dec(H,S,\mathcal{A}) \leq 
\dec(H_1,S,\bb),
\end{equation}
where, by definition,
$\dec(H_1,S,\bb)$ is the decay 
of the $n$-th minimal errors 
\begin{align*}
\err_n(H_1,S,\bb)
:=
\inf_{A \in \bb_n}
\{ \|S(f)-A(f)\|_{\mathcal{Z}} \colon
f\in H_1,\ \|f\|_{H_1}\leq 1\}
\end{align*}
for the corresponding univariate problem.

\subsection{Approximation with Unrestricted Linear Information}\label{s4.1}

We consider the abstract setting from Section \ref{s3.1}.
We are primarily interested in the case 
$\mathcal{X} = H$ and
\[
\mathcal{Z} := \bigotimes_{j\in\N} H_0,
\]
but for comparison we also consider $\mathcal{Z}$ together
with $\mathcal{X} = F$ or $\mathcal{X} = G$.
In all these cases the embedding operator $\App$, which
is well defined since $\alpha_{\nu,1} \geq 1\geq \gamma_j$ 
for $\nu,j \in \N$ yields embeddings
$G_j \hookrightarrow H_0$ of norm one for every $j \in \N$,
defines an infinite-dimensional approximation problem.

Let $\mathcal{X}^*$ denote the dual space of $\mathcal{X}$.
We are interested in approximating $\App$ 
using $n$ bounded linear functionals on $\mathcal{X}$, i.e., we
consider
\[
\aall_n := \left\{ \sum_{i=1}^n \lambda_i \cdot z_i
\colon \lambda_i \in \mathcal{X}^*,\ z_i \in \mathcal{Z} \right\}.
\]
Observe that in the standard setting we actually deal with 
$L_2$-approximation.

\begin{theo}\label{t10}
Consider the abstract setting, and assume that
\eqref{g5} is satisfied.
We have
\begin{align*}
\dec (H,\App,\aall)
& =\tfrac{1}{2} \cdot \min \left(
\decay \left((\alpha_{\nu,1}^{-1})_{\nu \in \N} \right),
\decay \left((\alpha_{1,j}^{-1})_{j \in \N} \right)
\right),\\
\dec (F,\App,\aall)
& =\tfrac{1}{2} \cdot
\decay \left((\alpha_{1,j}^{-1})_{j \in \N} \right),\\
\intertext{and}
\dec (G,\App,\aall)
& =
\tfrac{1}{2} \cdot \min \left(
\decay \left((\alpha_{\nu,1}^{-1})_{\nu \in \N} \right),
\decay \left((\gamma_j)_{j \in \N} \right)
\right).
\end{align*}
\end{theo}

\begin{proof}
First we consider $\dec (H,\App,\aall)$.
The singular values of $\App$ on $H$ are given by 
$\alpha_\bnu^{-1/2}$
with $\bnu \in \NN$, where
\[
\alpha_\bnu := \prod_{j\in \N} \alpha_{\nu_j,j},
\]
see \eqref{g400}.
Let $\xi := (\xi_i)_{i \in \N}$ denote the sequence of these
singular values, arranged in non-increasing order.
Due to a general result for linear problems on Hilbert spaces
\[
\err_n(H,\App,\aall) = \xi_{n+1},
\]
see \cite[Thm.~5.3.2]{TrWaWo88}.
Hence Lemma \ref{l1} and Lemma \ref{l4} yield
\begin{align*}
\dec(H,\App,\aall)
&=
\decay(\xi)\\
&= 
\sup \Bigl\{ \tau > 0 : \sum_{\bnu \in \NN}
\alpha_{\bnu}^{-1/(2\tau)} < \infty \Bigr\}\\
&= \tfrac{1}{2} \cdot \min \left(
\decay \left((\alpha_{\nu,1}^{-1})_{\nu \in \N} \right),
\decay \left((\alpha_{1,j}^{-1})_{j \in \N} \right)
\right).
\end{align*}

The results for $\dec(G,\App,\aall)$ and
$\dec(F,\App,\aall)$ are
established in the same way.
For the spaces $G$ we only have to observe that the singular
values of $\App$ are given by 
$\prod_{j\in\N} (\alpha_{\nu_j,1}/\gamma_j)^{-1/2}$
for $\bnu \in \NN$ and to use Remark \ref{r70} instead of Lemma
\ref{l4}.
The singular
values of $\App$ on $F$ are given by 
$\prod_{j\in\N} (\alpha_{\nu_j,j})^{-1/2}$
for $\bnu \in \NN\cap \{0,1\}^\N$, 
which immediately yields the claim.
\end{proof}

\begin{cor}\label{c1}
Consider the abstract setting.
For the Fourier weights according to Example \ref{ex1} we have
\begin{align}\label{eqqwert1}
\dec (H,\App,\aall)
=
\tfrac{1}{2} \cdot \min (r_1,
\rho \cdot \ln (a_1)),
\end{align}
and for the Fourier weights according to Example \ref{ex2a} we have
\begin{align}\label{eqqwert2}
\dec (H,\App,\aall)
=
\tfrac{1}{2} \cdot \rho \cdot \ln (a).
\end{align}
\end{cor}

\begin{proof}
Recall that 
\begin{equation}\label{g310}
\decay((\alpha_{\nu,1}^{-1})_{\nu \in \N}) = r_1
\end{equation}
and
\begin{equation}\label{g311}
\decay((\alpha_{1,j}^{-1})_{j \in \N}) = 
\decay((\gamma_{j})_{j \in \N}) = \rho \cdot \ln (a_1)
\end{equation}
for the first kind of Fourier weights,
and with $r_1= \infty$ and $a_1 = a$
we have the same decays 
for the second kind of Fourier weights,
see Examples \ref{ex1}, \ref{ex2a}, and \ref{exmp1}.
If $\rho > 0$, then \eqref{g5} is satisfied for both types
of Fourier weights,
and both of the claims follow from Theorem \ref{t10}.
If $\rho=0$ we get both of the claims from the fact that 
$\decay \left((\alpha_{1,j}^{-1})_{j \in \N} \right) = 0$.
\end{proof}

\begin{rem}
Let us also consider the finite-dimensional approximation
problem that is given by the embedding operator
\begin{align*}
\App\colon
H^{(d)}
:=\bigotimes_{j=1}^d H_j
\hookrightarrow
\bigotimes_{j=1}^d H_0.
\end{align*}

Determining $\dec (H,\App,\aall)$
or
determining the rate of strong polynomial tractability of
$\err_n(H^{(d)},\App,\aall)$
are equivalent problems.
More precisely,
\begin{align*}
\err_n(H,\App,\aall)
=\sup_{d \in \N}
\err_n(H^{(d)},\App,\aall),
\end{align*}
and hence
\begin{align}\label{eqwert}
\dec (H,\App,\aall) = 
\decay \left( 
\left(\sup_{d \in \N}
\err_n(H^{(d)},\App,\aall)\right)_{n \in \N} \right).
\end{align}
The quantity on the right hand side of \eqref{eqwert} is called 
the rate of strong polynomial tractability, and its reciprocal is 
called the exponent of strong polynomial tractability.
In this sense \eqref{eqqwert1} is due to
\cite[Thm.~1]{PapWoz10}, who study Korobov spaces and the case \eqref{g1},
where $a_1 = 2\pi$,
and \eqref{eqqwert2} is due to 
\cite[Thm.~5.2]{KriPiWo14}, who derive this result
for Korobov spaces
under the additional assumption of convergence 
of $(r_j / \ln(j))_{j \in \N}$; furthermore $\rho>0$ is established as a 
sufficient condition for strong polynomial
tractability. 

Our version of this result shows that $\rho>0$ is
also a necessary condition and thus settles an open problem from
\cite{KriPiWo14}.
\end{rem}

\begin{rem}\label{Rem_Dung_Griebel}
We consider a particular case of the setting
from \cite{DunGri16}, namely $m=1$ and $\beta=0$
in their notation. This means that the domain
is of the form $E := D_1 \times D \times D \times \dots$
with closed intervals $D_1,D \subseteq \R$
and that the space $H_1$ may be defined
in terms of an orthonormal basis that is different from the basis
used to defined the other spaces $H_j$ with $j \geq 2$.
Furthermore, $\App$ maps $H$ into the $L_2$-space with
respect to a product probability measure of the form
$\mu_1 \times \mu_0 \times \mu_0 \times \dots$. Our results extend
to this setting in a straight-forward way.

In \cite{DunGri16} the Fourier weights from Example \ref{ex1}
with \eqref{g2} are considered. It is shown that
\begin{equation}\label{g52}
\dec (H,\App,\aall) = \tfrac{1}{2} \cdot r_1
\end{equation}
holds, if the requirement
\begin{equation}\label{g51}
\sum_{j \in \N}\frac{1}{r_j} \cdot \eta^{r_j} < \infty
\end{equation}
for $\eta = (2/3)^{1/r_1}$ is satisfied,
see \cite[Thm.~4.1]{DunGri16}.
Observe that 
\[
\rho \cdot \ln (3/2) >  r_1
\]
is a sufficient condition for \eqref{g51} to hold, and a necessary
condition for \eqref{g51} also permits equality, cf. Lemma~\ref{l5} 
of Appendix B.

Notice that \eqref{g2} implies $a_1=2$. 
Thus our result \eqref{eqqwert1} improves on the findings in 
\cite{DunGri16}
as it shows that the weaker condition
\[
\rho \cdot \ln (2)  \geq r_1
\]
is necessary and sufficient for \eqref{g52} to hold.
Nevertheless, we stress that in
\cite{DunGri16} explicit error bounds are derived, while our Theorem 
\ref{t10}
only determines the decay of the minimal errors. 
\end{rem}

\subsection{Approximation and Integration
with Standard Information}\label{sec42}

Now, we investigate the approximation and the integration problem 
in the standard setting.

In the sequel $\mathcal{X}$ typically
will be one of the tensor product spaces 
$F$, $H$, $G$, $F^c$, or $G^c$. For approximation we have
\[
\mathcal{Z} := \bigotimes_{j \in \N} L_2(D,\mu_0) = L_2(D^\N,\mu)
\]
with the respective embedding $S = \App$. For integration we have
$\mathcal{Z} = \K$, and $S= \Int$ is given by
\[
\Int(f) = \int_{D^\N} f \, d \mu
\]
for $f \in \mathcal{X}$.

Again, we are primarily interested in the case
$\mathcal{X} = H$, but in our analysis it is crucial to
consider $\mathcal{X} = F^c$ and $\mathcal{X} = G^c$ with
suitably chosen $c>0$ as well. The latter enables us to 
apply general results from \cite{PW11,Was12} for weighted tensor product
spaces based on anchored kernels. For comparison we
also consider $\mathcal{X} = F$ and $\mathcal{X} = G$ as before. 

Assume that $\mathcal{X}$ is a reproducing kernel Hilbert space of
functions on the domain $D^\N$, so that 
$\delta_{\by} (f) := f(\by)$ defines a bounded linear 
functional on $\mathcal{X}$ for every $\by \in D^\N$.
We consider a class $\astd_n$ of bounded linear operators
that is much smaller than $\aall_n$.
First of all,
$A \in \astd_n$ is only based on function
evaluations $\delta_{\by}$ instead of arbitrary bounded linear
functionals $\lambda \in \mathcal{X}^*$. Furthermore, we
do not permit evaluations at any point $\by \in D^\N$ and 
also do not just take into account the total number of function
evaluations that is used by $A$.
Instead, we employ the unrestricted subspace sampling
model, which has been introduced in \cite{KuoEtAl10}. 
This model is based on a non-decreasing cost function 
$\$\colon \N_0 \cup \{\infty\} \to [1,\infty]$ and some 
nominal value $a\in D$ in the following way.
For $\by = (y_j)_{j \in \N} \in D^\N$ the number of active variables is
given by
\begin{align}\label{yaqwsx1}
\Act(\by)
:=\#\{j\in\N\colon y_j\neq a\},
\end{align}
and
\[
\astd_n := 
\left\{ \sum_{i=1}^m \delta_{\by_i} \cdot z_i \colon 
m\in\N_0,\,
\by_i\in D^\N,\, 
\sum_{i=1}^m \$(\Act(\by_i))\leq n,\,
z_i \in \mathcal{Z} \right\}
\]
is the class of algorithms with the cost bounded by $n$. 
For the univariate approximation and integration problems, where
$\mathcal{Z} := L_2(D,\mu_0)$ or $\mathcal{Z} := \K$ and $\Int(f)
:= \int_D f \, d \mu_0$, respectively,
we simply take as sets $\bstd_n$ of bounded linear 
operators between $H_1$ and $\mathcal{Z}$
\[
\bstd_n := 
\left\{ \sum_{i=1}^n \delta_{y_i} \cdot z_i \colon 
y_i\in D,\, z_i \in \mathcal{Z} \right\}.
\]

In the sequel, we assume that 
\begin{align}\label{yaqwsx2}
\$(n) = \Omega(n)\text{ and }
\$(n)= O(e^{\zeta n})\text{ for some }\zeta\in (0,\infty).
\end{align}

\begin{theo}\label{thmapp}
Consider the standard setting, and assume that $G$ is a reproducing
kernel Hilbert space. 
For $S=\App$ and $S=\Int$ we have
\begin{equation}\label{eq1}
\dec (F,S,\astd) =
\tfrac{1}{2} \cdot \left(\decay((\alpha_{1,j}^{-1})_{j\in\N})-1\right)
\end{equation}
and
\begin{equation}\label{eq2}
\dec(G,S,\astd)
= 
\min\left(
\dec(H_1,S,\bstd),\tfrac{1}{2}
\cdot \left(\decay((\gamma_{j})_{j\in\N})-1\right)\right),
\end{equation}
implying
\begin{align}\label{eq3}
&\min\left(
\dec(H_1,S,\bstd),\tfrac{1}{2}
\cdot \left(\decay((\gamma_{j})_{j\in\N})-1\right)\right)\\
& \qquad \le \dec(H,S,\astd) \notag \\
& \qquad \le \min \left( \dec(H_1, S,\bstd),  
\tfrac{1}{2} \cdot \left( \decay((\alpha_{1,j}^{-1})_{j\in\N}) -1 \right) 
\right).
\notag
\end{align}
\end{theo}

\begin{proof}
At first, we derive \eqref{eq2}. For every $c>0$ the reproducing
kernel of $G^c$ is a weighted tensor product that is based on an
anchored kernel, see Lemma~\ref{emb1dimg}. We claim that
\[
\dec(G^c,S,\astd) =
\min\left(
\dec(H_1,S,\bstd),\tfrac{1}{2}
\cdot \left(\decay((\gamma_{j})_{j\in\N})-1\right)\right).
\]
In fact, if $\decay((\gamma_{j})_{j\in\N})>1$, then 
\cite[Cor.~9]{Was12} yields this claim for $S=\App$, while we
employ \cite[Thm.~2 and Sec.~3.3]{PW11} for $S=\Int$.
Otherwise we have $\decay((\gamma_{j})_{j\in\N})=1$, see 
\eqref{sumfinite}, and this case 
may be easily reduced to the 
previous one. Indeed, this can be done by making the weights smaller 
such that their decay $\delta$ is strictly larger than one. Making 
the weights smaller leads to smaller $n$-th minimal errors and thus
to a larger decay of the minimal errors. Using the claim for the 
case $\decay((\gamma_{j})_{j\in\N})>1$ and letting $\delta$ tend 
to one establishes our claim $\dec(G^c, S, \astd)=0$ in the case
$\decay((\gamma_{j})_{j\in\N})=1$.

With the help of our claim and the embedding result from
Lemma~\ref{emb1dimg}, we obtain \eqref{eq2}.

The proof of \eqref{eq1} is similar.
Here we only have to observe that
$e_0=1$ and $e_1(x)\neq e_1(y)$ for some $x,y\in D$, which
yields $\err_2(F_1,S,\bstd) = 0$, and to apply Lemma \ref{emb1dimf}
instead of Lemma \ref{emb1dimg}.

Finally, \eqref{eq3} follows from
\eqref{eq1} and \eqref{eq2} together with \eqref{g300}. 
\end{proof}

\begin{rem}
In the proof of Theorem \ref{thmapp} we rely on results from
\cite{PW11, Was12} that were actually proved under slightly 
stronger assumptions than the ones we make in the theorem. 
It is assumed in  \cite{PW11, Was12} that
$D$ is a Borel measurable subset of $\R$ and
that the probability measure $\mu_0$ 
has a Lebesgue density.
The proofs are applicable, however, in the setting of the present
paper, cf.\ \cite{DG14,GMR14}.
\end{rem}

We apply Theorem \ref{thmapp} to the trigonometric basis
and to the Haar basis.
For the univariate problem on the corresponding space $H_1$
the asymptotic behavior of the $n$-th minimal errors
$\err_n(H_1,S,\bstd)$ is known for $S=\App$ and $S = \Int$
in the case of the trigonometric basis.
In the case of the Haar basis we are only aware of a lower bound
for $S = \Int$. A matching upper bound for $S=\App$ is established
in Appendix \ref{a3}.

\begin{cor}\label{c2}
Assume that $H_0 = L_2([0,1],\mu_0)$ for the uniform distribution
$\mu_0$.
Consider the trigonometric or the Haar basis
$(e_\nu)_{\nu \in \N_0}$ according to 
Example~\ref{ex11} or Example~\ref{eha}, respectively.
For the Fourier weights according to Example~\ref{ex1} we have
\[
\bigl(r_1 > 1 \wedge \rho \cdot \ln (a_1) > 1\bigr)
\Rightarrow \text{$H$ is a RKHS}
\Rightarrow
\bigl(r_1 > 1 \wedge \rho \cdot \ln (a_1) \geq 1\bigr),
\]
as well as
\[
\dec(H,S,\astd) = \tfrac{1}{2}\cdot \min(r_1,\rho\cdot \ln(a_1)-1)
\]
for $S=\App$ and $S = \Int$
if $H$ is a reproducing kernel Hilbert space.
\end{cor}

\begin{proof}
Examples~\ref{ex11}, \ref{eha}, and \ref{ex1} with $\sigma=0$
yield the necessary and the sufficient condition for
$H$ to be a reproducing kernel Hilbert space.
Recall that $G$ is based on the Fourier weights
$(\alpha_{\nu,1}/\gamma_j)_{\nu,j \in \N}$. 
We proceed as for the space $H$ to establish the same pair of 
conditions for $G$ to be a reproducing kernel Hilbert space.

In the sequel we therefore assume that $r_1 > 1$. Then we have
\begin{align}\label{deconedim}
\dec(H_1,S,\bstd) =\tfrac{1}{2} \cdot r_1.
\end{align}
The lower bound for $S=\App$ in the case of the
trigonometric basis follows from the well-known approximation
error estimates of Dirichlet or de la Vall\'ee-Poussin means,
see, e.g., \cite{Khr83, Tem93}. The case of the Haar basis is
studied in Theorem \ref{t95}.
The upper bound for $S=\Int$ in the the case of the
trigonometric basis follows from equally well-known 
constructions of fooling functions that are trigonometric 
polynomials, see, e.g., \cite{Tem90}.
The case of the Haar basis follows from \cite[Thm.~41]{MR3330331}.

If $\rho \cdot \ln (a_1) > 1$ is valid, then we apply \eqref{g311}
and \eqref{eq3} to determine $\dec(H,S,\astd)$ as claimed,
and the remaining case $\rho \cdot \ln (a_1) = 1$ is easily reduced
to the previous one. 
\end{proof}

In a similar way we may handle the Fourier weights according to
Example~\ref{ex2a} instead of Example~\ref{ex1}. 
Since this type of Fourier weights does never satisfy \eqref{condpower},
we only consider the trigonometric basis.

\begin{cor}\label{c3}
Assume that $H_0 = L_2([0,1],\mu_0)$ for the uniform distribution
$\mu_0$. Consider the trigonometric basis $(e_\nu)_{\nu \in \N_0}$ 
according to Example~\ref{ex11}.
For the Fourier weights according to Example~\ref{ex2a} we have
\[
\rho \cdot \ln (a) > 1
\Rightarrow \text{$H$ is a RKHS}
\Rightarrow
\rho \cdot \ln (a) \geq 1,
\]
as well as
\[
\dec(H,S, \astd) = \tfrac{1}{2}\cdot (\rho\cdot \ln(a)-1)
\]
for $S=\App$ and $S = \Int$
if $H$ is a reproducing kernel Hilbert space.
\end{cor}

\begin{proof}
Proceed as in the proof of Corollary \ref{c2}.
\end{proof}

\subsection{Concluding Remarks}

\begin{rem}\label{Rem_when_does...}
Consider the approximation or the integration problem in the
standard setting, and assume that $G$ is a reproducing kernel
Hilbert space.
For the trigonometric basis and the polynomial or (sub-)exponential
Fourier weights and for the Haar basis and the polynomial 
Fourier weights we obtain sharp results
via embeddings: It turns out that
\begin{align*}
\dec(G^{c_0^{-1}},S,\astd)
&=
\dec(G,S,\astd) =
\dec(H,S,\astd)\\
&=\min\left(\dec(H_1,S,\astd),\dec(F,S,\astd)\right)
\\
&=\min\left(\dec(H_1,S,\astd),\dec(F^{c_0},S,\astd)\right)
\end{align*}
for $S=\App$ and $S = \Int$,
see Corollaries~\ref{c2} and~\ref{c3} and their proofs. 

The analysis on the tensor product $H$ of spaces of increasing smoothness
is therefore reduced to the analysis on its first factor $H_1$
and on the weighted tensor product spaces $G^{c_0^{-1}}$ and
$F^{c_0}$, which are based on anchored kernels.
We stress that already the space $G$ is typically much larger than $H$,
while already the space $F$ is always much smaller than $H$.

A similar conclusion holds true for the approximation problem 
with unrestricted linear information in the
abstract setting.
For the polynomial and the {(sub-)}ex\-ponential Fourier weights 
\begin{align*}
\dec(G,\App,\aall)
&=\dec(H,\App,\aall)\\
&=\min\left(\dec(H_1,\App,\aall),\dec(F,\App,\aall)\right),
\end{align*}
see Theorem \ref{t10}, Corollary~\ref{c1} as well as \eqref{g310} and
\eqref{g311}.
\end{rem}

\begin{rem}\label{exaappstd}
Consider the approximation problem in the setting from  
Corollary \ref{c2}.
Combining the latter with Corollary \ref{c1} 
reveals that, at least with respect to the decay of the $n$-th
minimal errors, the class $\astd_n$ is as powerful as the class
$\aall_n$ if and only if $\rho\cdot \ln(a_1)\geq r_1+1$.
Furthermore, $\dec(H,\App,\astd)$ and $\dec(H,\App,\aall)$ differ
at most by $1/2$. 

In the setting from Corollary \ref{c3} we always have
\begin{align*}
\dec(H,\App,\astd)
=\dec (H,\App,\aall)-\tfrac{1}{2}.
\end{align*}
\end{rem}

\appendix

\section{Countable Tensor Products}\label{a1}

Let $(H_j,\scp{\cdot}{\cdot}_j)_{j \in \N}$ be a sequence of Hilbert 
spaces and fix, for each $j\in\N$, a
unit vector $u_j \in H_j$.
If it is clear from the context we sometimes omit to name the unit 
vectors $u_j \in H_j$ explicitly. In the 
setting of Section \ref{s3a} it is natural to choose 
$u_j=e_0$ for all $j\in\N$. Then the incomplete tensor product 
\[ 
 H := \bigotimes_{j\in \N} H_j
\]
is the completion of the linear span of elementary infinite tensors 
$ \otimes_{j \in \N}f_j$, for which only finitely many $f_j$ are different 
from $u_j$. Here the completion is taken with respect to the inner 
product given by
\[ 
 \scp{\otimes_{j\in \N} f_j}{\otimes_{j\in \N} g_j} := 
 \prod_{j \in \N} \scp{f_j}{g_j}_j
\]
for elementary infinite tensors and extended linearly to finite sums of 
elementary infinite tensors. The abstract completion can be replaced by 
a concrete description of elements in the incomplete tensor product via 
linear functionals, see \cite{Neu39}.

This notion of an infinite tensor product is the natural one for our 
purpose since the incomplete tensor product of spaces $L_2(D_j,\mu_j)$ 
is in a canonical way isometrically isomorphic
to $L_2(D, \mu)$, where 
$\mu:=\times_{j \in \N} \mu_j$ is the product measure of the probability 
measures $\mu_j$ on $ D := \times_{j\in \N} D_j$.

We freely used the following facts, which can be found in \cite{Neu39}. 
Each $H_{j_0}$ is isometrically embedded in $H$ by identifying 
$h_{j_0} \in H_{j_0}$ with the tensor $\otimes_{j \in \N} f_j$ with 
$f_{j_0} = h_{j_0}$ and $f_j=u_j$ for $j \neq j_0$. Similarly, the 
finite Hilbert space tensor products $\bigotimes_{j=1}^d H_j$ are 
isometrically embedded in $H$. 
If we have another incomplete tensor product 
\[ 
 G := \bigotimes_{j \in \N} G_j
\]
with unit vectors $v_j\in G_j$
and a sequence of bounded linear operators $T_j:H_j \to G_j$
with $T_j u_j = v_j$ such that 
\[
C:=\prod_{j \in \N} \|T_j\| < \infty,
\]
then there exists a unique 
linear bounded operator $T:H \to G$ acting on elementary tensors as
\[
 T \left( \otimes_{j \in \N} f_j \right) := 
 \otimes_{j \in \N} T_j f_j. 
\] 
Moreover, $\|T\|=C$.

\section{Summability and Decay of Sequences}\label{a2}

As in Section \ref{s2.2} we consider 
sets $N_j \subseteq \N_0$ such that $0 \in N_j$ for $j \in \N$, and we let
$\NN$ denote the set of all sequences $\bnu := (\nu_j)_{j \in
\N}$ in $\N_0$ such that $\nu_j \in N_j$ for every $j \in \N$
and $\sum_{j \in \N} \nu_j < \infty$.

\begin{lemma}\label{l1}
Let 
$\beta_\bnu := \prod_{j \in \N} \beta_{\nu_j,j}
$
for $\bnu \in \NN$ and $\beta_{\nu,j} \in \R$ for $j \in \N$ and
$\nu \in N_j$ with
$\beta_{0,j} = 1$ for every $j \in \N$. Then 
$\sum_{\bnu \in \NN} \beta_{\bnu}$ is absolutely convergent if and
only if 
\[
\sum_{j \in \N} \sum_{\nu \in N_j \setminus \{0\}}
|\beta_{\nu,j}| < \infty,
\]
in which case
\[
\sum_{\bnu \in \NN} \beta_{\bnu}
=
\prod_{j \in \N} \left(1 + \sum_{\nu \in N_j \setminus \{0\}}
\beta_{\nu,j} \right).
\]
\end{lemma}

\begin{proof}
Without loss of generality we may assume 
that $N_j = \N_0$ for every $j \in \N$.
Let
\[
\NN_k := \{ \bnu \in \NN : \text{$\nu_j=0$ for $j > k$}\}
\]
for $k \in \N$.
It is easy to prove by induction that
\[
\sum_{\bnu\in\NN_k}|\beta_\bnu|=
\prod_{j=1}^k \left(1+\sum_{\nu \in \N}|\beta_{\nu,j}|\right)
\]
for every $k \in \N$. Therefore
\[
\sum_{\bnu\in\NN}|\beta_\bnu|
=
\prod_{j\in\N}\left(1+\sum_{\nu \in \N}|\beta_{\nu,j}|\right),
\]
so that $\sum_{\bnu\in\NN}|\beta_\bnu| < \infty$ and
$\sum_{\nu,j \in \N}|\beta_{\nu,j}| < \infty$ are equivalent.
Similarly, we get 
\[
\sum_{\bnu\in\NN}\beta_\bnu =
\prod_{j\in\N}\left(1+\sum_{\nu\in\N}\beta_{\nu,j}\right), 
\]
if $\sum_{\bnu\in\NN}|\beta_\bnu| < \infty$. 
\end{proof}

\begin{lemma}\label{l4}
Assume that
\eqref{g5} is satisfied, 
in addition to \eqref{g6b} and \eqref{g6a}.
For every $\tau > 0$ and every $\sigma \geq 0$ we have
\[
\sum_{\nu,j \in \N} \alpha_{\nu,j}^{-\tau} \cdot \nu^\sigma< \infty
\quad \Leftrightarrow \quad
\left(
\sum_{\nu \in \N} \alpha_{\nu,1}^{-\tau} \cdot \nu^\sigma < \infty
\ \wedge \ 
\sum_{j \in \N} \alpha_{1,j}^{-\tau} < \infty
\right).
\]
\end{lemma}

\begin{proof}
It suffices to verify the implication `$\Leftarrow$'
for $\tau=1$.
Accordingly, we assume that 
$\sum_{\nu \in \N} \alpha_{\nu,1}^{-1} \cdot \nu^\sigma < \infty$ and
$\sum_{j \in \N} \alpha_{1,j}^{-1} < \infty$, and we show
that
\[
\sum_{\nu,j \in \N} \alpha_{\nu,j}^{-1} \cdot \nu^\sigma < \infty.
\]
Put $\beta_{\nu,j} := \alpha_{\nu,j} \cdot \nu^{-\sigma}$.
For any choice of $n \in \N$ we have
\begin{align*}
\sum_{\nu,j \in \N} \beta_{\nu,j}^{-1}
&\leq 
\sum_{j=1}^{n} \sum_{\nu \in \N} \beta_{\nu,j}^{-1} 
+
\sum_{\nu=1}^{n} \sum_{j \in \N} \beta_{\nu,j}^{-1} 
+
\sum_{\nu=n+1}^\infty \sum_{j=n+1}^ \infty
\beta_{\nu,j}^{-1} \\
&\leq 
n \cdot \sum_{\nu \in \N} \alpha_{\nu,1}^{-1} \cdot
\nu^\sigma
+
n^{\sigma+1} \cdot
\sum_{j \in \N} \alpha_{1,j}^{-\tau} 
+
\sum_{\nu=n+1}^\infty \sum_{j=n+1}^ \infty
\beta_{\nu,j}^{-1}, 
\end{align*}
see \eqref{g6b}, where the two single sums are finite by assumption.
Choose $\eps > 0$ and $n \geq \exp(2/\eps)$ such that
\[
\ln(\beta_{\nu,j}) \geq \eps \cdot \ln (\nu) \cdot \ln(j)
\]
for all $\nu,j \geq n$, see \eqref{g5}.
For $\nu,j$ as before we obtain
\[
\beta_{\nu,j}^{-1} =
\exp (- \ln (\beta_{\nu,j}))
\leq 
\exp (- \eps \cdot \ln (\nu) \cdot \ln(j)).
\]
Hence
\[
\sum_{j=n+1}^\infty \beta_{i,j}^{-1} \leq
\sum_{j = n+1}^\infty j^{-\eps \cdot \ln (i)}
\leq n^{-\eps\cdot\ln(i)+1}
\]
for every $i \geq n$, and therefore
\[
\sum_{i=n+1}^\infty
\sum_{j=n+1}^\infty \beta_{i,j}^{-1} \leq
n \cdot 
\sum_{i=n+1}^\infty
i^{-\eps \cdot \ln (n)} < \infty.
\qedhere
\]
\end{proof}

\begin{lemma}\label{l5}
Let $q_j >0$ for $j \in \N$ and
\[
q := \liminf_{j \to \infty} \frac{q_j}{\ln(j)}. 
\]
Then we have
\[
q > 1 \quad \Rightarrow \quad
\sum_{j \in \N} \exp(-q_j) < \infty
\quad \Rightarrow \quad
q \geq 1.
\]
\end{lemma}

\begin{proof}
Assume that $0 \leq q < 1$. 
Then there
exists a  sequence of integers $j_m$
such that $q_{j_m} \leq \ln(j_m)$ and $j_{m+1} \ge 2 j_m$.
Consequently,
\[
\sum_{j\in\N} \exp(-q_j) \geq
\sum_{m=2}^\infty \exp(-q_{j_m}) \cdot (j_m - j_{m-1})
\geq
\sum_{m=2}^\infty j_m^{-1} \cdot (j_m - j_{m-1}) =\infty,
\] 
where in the last step we used that $1 - j_{m-1}/j_m > 1/2$ for all $m\ge 2$.

Now we assume that $q > 1$.
Choose $\eps > 0$ such that $1+\eps < q$ and
$j_0 \in \N$ such that $q_j \geq (1+\eps) \cdot \ln (j)$
for every $j \geq j_0$. Then we  get
\[
\sum_{j \in \N} \exp(-q_j) \leq
j_0 -1 +
\sum_{j = j_0}^\infty j^{-(1+\eps)} < \infty.
\qedhere
\]
\end{proof}

\begin{exmp}
Consider the limiting case $q=1$ in Lemma \ref{l5}.
For $q_j = \ln(j)$
we have $q=1$ and
\[
\sum_{j \in \N} \exp(-q_j) = \sum_{j \in \N} j^{-1} = \infty.
\]
For $q_j = \ln(j) + 2\ln\ln(j)$
we have $q=1$ as well, but
\[
\sum_{j \in \N} \exp(-q_j) = \sum_{j \in \N} j^{-1}
\cdot (\ln(j))^{-2}< \infty.
\]
\end{exmp}

\section{$L_2$-Approximation in Haar Spaces}\label{a3}

For $n \in \N_0$ we put
$M := \{0,\dots,2^n-1\}$, and for $m \in M$ we consider the intervals
\[
\phantom{\qquad\quad m < 2^n-1,}
I_m := {[m/2^n,(m+1)/2^n[},
\qquad\quad m < 2^n-1,
\]
as well as
\[
I_{2^n-1} := [(2^n-1)/2^n,1]. 
\]
Moreover,
let $T_n(f)$ be the piecewise constant interpolation of 
$f \colon [0,1] \to \C$ on these intervals, based on the values of
$f$ at the respective midpoints.

\begin{theo}\label{t95}
Assume that $H_0 = L_2([0,1],\mu_0)$ for the uniform distribution
$\mu_0$. Consider the Haar basis $(e_\nu)_{\nu \in \N_0}$ according to 
Example~\ref{eha} and the Fourier weights according to Example~\ref{ex1}.
Furthermore, assume that $H_1$ is a reproducing kernel Hilbert space, 
i.e., $r_1>1$. Then there exists a constant $C> 0$ such that 
\begin{align}\label{eq0appc}
\sup\{\|f-T_n(f)\|_0\colon f\in H_1,\ \|f\|_{1}\leq 1\}
\leq C\cdot {2^{-n \cdot r_1/2}}
\end{align}
for all $n\in\N_0$. Furthermore, 
\begin{align}\label{eqappc}
\err_n(H_1,S,\bstd)
\asymp
\err_n(H_1,S,\aall)
\asymp
n^{-r_1/2}
\end{align}
for the embedding $S\colon H_1\hookrightarrow H_0$.
\end{theo}

\begin{proof}
Fix $n\in\N_0$, and let $f:=\sum_{\nu\in\N_0}a_\nu e_\nu$ with 
$a_\nu\in\C$ such that $a_\nu \neq 0$ for only finitely many 
$\nu \in \N_0$. Put 
\[
k(\ell,m):=
2^\ell+
m2^{\ell-n}
+
\begin{cases}
0, & \text{if $\ell=n$},\\
2^{\ell-n-1}, & \text{if $\ell>n$},
\end{cases}
\]
as well as
\[
c(\ell):=
2^{\ell/2}\cdot
\begin{cases}
-1, & \text{if $\ell=n$},\\
+1, & \text{if $\ell>n$},
\end{cases}
\]
for $\ell \geq n$ and $m \in M$.
Observe that 
\begin{align}\label{eqtnu}
\begin{aligned}
T_n(e_\nu) = 
\begin{cases}
e_\nu, & \text{if $\nu \leq 2^n-1$},\\
c(\ell) \cdot 1_{I_m}, & \text{if $\nu=k(\ell,m)$ for $\ell\geq n$ and $m
\in M$},\\
0, & \text{otherwise}.
\end{cases}
\end{aligned}
\end{align}
>From \eqref{eqtnu} we get
\begin{equation}\label{eq1appc}
\|f-T_n(f)\|_0
\leq \bigl\|\sum_{\nu\geq 2^n}a_\nu e_\nu\bigr\|_0
+\bigl\|\sum_{\nu\geq 2^n}a_\nu T_n(e_\nu)\bigl\|_0,
\end{equation}
and obviously
\begin{align}\label{eq2appc}
\bigl\|\sum_{\nu\geq 2^n}a_\nu e_\nu\bigl\|_0^2
=\sum_{\nu\geq 2^n}|a_\nu|^2
\leq 2^{-n \cdot r_1}
\sum_{\nu\in\N_0} |a_\nu|^2 \nu^{r_1}.
\end{align}
Furthermore, \eqref{eqtnu} yields
\begin{align*}
\sum_{\nu\geq 2^n}a_\nu T_n(e_\nu)
=
\sum_{m \in M} \sum_{\ell \geq n} a_{k(\ell,m)} T_n(e_{k(\ell,m)})
=
\sum_{m \in M} \Bigl(\sum_{\ell \geq n} a_{k(\ell,m)} \, c(\ell)
\Bigr) \cdot  1_{I_m}.
\end{align*}
Note that 
$\scp{1_{I_{m}}}{1_{I_{m^\prime}}}_0$=0 for $m,m^\prime \in M$
with $m \neq m^\prime$. Therefore
\begin{align*}
\bigl\|\sum_{\nu\geq 2^n}a_\nu T_n(e_\nu)\bigr\|_0^2
= \sum_{m\in M}\Bigl|\sum_{\ell \geq n} a_{k(\ell,m)} \, c(\ell)
\Bigr|^2 2^{-n}.
\end{align*}
Hence the Cauchy-Schwarz inequality and the fact 
$k(\ell,m)\geq 2^\ell$ yield
\begin{align*}
\bigl\|\sum_{\nu\geq 2^n}a_\nu T_n(e_\nu)\bigr\|_0^2
&\leq 2^{-n}\sum_{m\in M}
\Bigr(
\sum_{\ell\geq n}|a_{k(\ell,m)}|^2 {|k(\ell,m)|}^{r_1} \cdot
\sum_{\ell\geq n}{|k(\ell,m)|}^{-r_1} 2^\ell\Bigl)\\
&\leq 2^{-n}\sum_{\ell\geq n}2^{\ell (1-r_1)}
\cdot
\sum_{\nu\in\N_0} |a_\nu|^2 \nu^{r_1}\\
&= c^2 \cdot
2^{-n \cdot r_1}\sum_{\nu\in\N_0}|a_\nu|^2 \nu^{r_1}
\end{align*}
with $c = (1-2^{1-r_1})^{-1/2}$.
Combing this estimate with \eqref{eq1appc} and \eqref{eq2appc} yields
\begin{align*}
\|f-T_n(f)\|_0\leq
(1+c) \cdot 2^{-n\cdot r_1/2}
\bigg(\sum_{\nu\in\N_0}|a_\nu|^2 \nu^{r_1}\bigg)^{1/2},
\end{align*}
which shows \eqref{eq0appc}.
It remains to observe that
\begin{align*}
\err_n(H_1,S,\bstd)
\geq
\err_n(H_1,S,\aall)
\succeq
n^{-r_1/2}
\end{align*}
to complete the proof of \eqref{eqappc}.
\end{proof}

\subsection*{Acknowledgment}
The authors would like to thank Michael Griebel and Henryk Wo\'zniakowski 
for interesting discussions and valuable comments. 

The work on this paper was initiated during a special semester at the 
Institute for Computational and Experimental Research in Mathematics 
(ICERM) of Brown University. Part of the work was done during a
visit of the University of New South Wales (UNSW) in Sydney and 
during a special semester at the Erwin Schr\"odinger International 
Institute for Mathematics and Physics (ESI) in Vienna.

M.~Hefter was supported by the Austrian Science Fund (FWF), 
Project F5506-N26.
A.~Hinrichs is supported by the Austrian Science Fund (FWF), Project 
F5509-N26. Both projects are part of the Special Research Program 
``Quasi-Monte Carlo Methods: Theory and Applications''.
K.~Ritter was partially supported as a visiting professor
at Kiel University.

\bibliographystyle{siam}
\bibliography{references}

\end{document}